\documentclass[a4paper,onesided]{amsart} %\Bbb R_+
\usepackage[T1]{fontenc}
\usepackage[utf8]{inputenc}
\usepackage{amsfonts}
\usepackage{amsbsy}
\usepackage[left=2.25cm,right=2.25cm,top=4cm,bottom=4cm]{geometry}
\usepackage{color}
\usepackage{enumerate}
\usepackage{upgreek}
\usepackage{color}
\usepackage{amsmath}   % many want amsmath extensions
\usepackage{amsfonts,amssymb, color}
\usepackage{latexsym}
\input colordvi
\usepackage{amsfonts}

\usepackage{ulem}
\usepackage{mathrsfs}
\usepackage[usenames,dvipsnames]{xcolor}

\newcommand\coc[1]{{\overline{#1}}} %30/05/2012

\newcounter{gr1}

\newcounter{gr1n}
\newenvironment{numlistn}
{\begin{list} { (\roman{gr1n})} {\usecounter{gr1n}
\setlength{\leftmargin}{0.9cm}
\setlength{\topsep}{0.1cm} \setlength{\itemsep}{0.0cm}
\setlength{\parsep}{0.1cm} \setlength{\itemindent}{-0.7cm}
\setlength{\parskip}{0.0cm}}} {\end{list}}

\newcommand{\Int}{(0,T]}
\newcommand{\INT}{[0,T]}
\newcommand{\MA}{{\mathfrak{A}}}
\newcommand{\AAA}{{U}}

\newcommand{\Law}{{\mathcal{L}aw}}
\newcommand{\fh}{{\mathfrak{h}}}

\newcommand{\fhs}{\mathfrak{h} ^ s}

\newcommand{\CMM}{{{M}_{\bar \NN}( \{S_n\times \Int \})}}

\newcommand{\CX}{{{ \mathcal X }}}

\newcounter{lil}

\newcounter{gr111}
\newenvironment{steps}
{\begin{list} { (\roman{gr111})$\, $} {\usecounter{gr111}
\setlength{\labelwidth}{-0.2cm} \setlength{\leftmargin}{0.4cm}
\setlength{\topsep}{0.1cm} \setlength{\itemsep}{0.2cm}
\setlength{\parsep}{0.1cm} \setlength{\itemindent}{0.4cm}
\setlength{\parskip}{0.0cm}}} {\end{list}}

\newcounter{gr11}
\newenvironment{letters}
{\begin{list} { (\alph{gr11})$\, $} {\usecounter{gr11}
\setlength{\labelwidth}{-0.2cm} \setlength{\leftmargin}{0.4cm}
\setlength{\topsep}{0.1cm} \setlength{\itemsep}{0.2cm}
\setlength{\parsep}{0.1cm} \setlength{\itemindent}{0.4cm}
\setlength{\parskip}{0.0cm}}} {\end{list}}

\newcounter{gr1r}

\newcommand{\bNN}{{\bar{ \mathbb{N}}}}

\newcommand{\baray}{\begin{array}{rcl}}
\newcommand{\earay}{\end{array}}
\newcommand{\barray}{\begin{array}{rcl}}
\newcommand{\earray}{\end{array}}

\newcommand{\bcal}{\mathcal{B}}

\newcommand{\mG}{\mathfrak{G}}
\newcommand{\mF}{\mathfrak{F}}
\newcommand{\mT}{\mathfrak{T}}
\newcommand{\mS}{\mathfrak{S}}

\newcommand{\levy}{L\'evy }

\newcommand\blufy[1]{{\color{blue} #1}}
\newcommand\refy[1]{{}}% \color{red}\sout {#1}}}
\newcommand\rmfy[1]{{\color{blue} %\sout
 {#1}}}

\newcommand\dela[1]{}

\newcommand{\bcase}{\begin{cases}}
\newcommand{\ecase}{\end{cases}}

\newcommand\cadlag{c{\'a}dl{\'a}g }

\newcommand\del[1]{}

\del{

\newcommand\del[1]{}
}

\newcommand{\lk}{\left}
\newcommand{\lqq}{\lefteqn}
\newcommand{\rk}{\right}
\newcommand{\la}{{\langle}}
\newcommand{\ra}{{\rangle}}

\newcommand{\LL}{{\rm I \kern -0.2em L}}

\newcommand{\ep} {\varepsilon }

\newcommand{\be} {\begin{enumerate} }
\newcommand{\ee} {\end{enumerate} }

\newcommand{\CO}{{{ \mathcal O }}}
\newcommand{\CK}{{{ \mathcal K }}}

\newcommand{\CC}{{{ \mathcal C }}}

\newcommand{\FG}{{{\mathfrak{G} }}}

\newcommand{\CZ}{{{ \mathcal S }}}
\newcommand{\CT}{{{ \mathcal T }}}
\newcommand{\CI}{{{ \mathcal I }}}

\newcommand{\CH}{{{ \mathcal H }}}
\newcommand{\CS}{{{ \mathcal S }}}

\newcommand{\CB}{{{ \mathcal B }}}

\newcommand{\CM}{{{ \mathcal M }}}
\newcommand{\CP}{{{ \mathcal P }}}
\newcommand{\BF}{{{ \mathbb{F} }}}
\newcommand{\CF}{{{ \mathcal F }}}

\newcommand{\CN}{{{ \mathcal N }}}
\newcommand{\CL}{{{ \mathcal L }}}

\newcommand{\RR}{{\mathbb{R}}}
\newcommand{\Rb}[1]{{\mathbb{R}_{#1}}}

\newcommand{\DD}{\mathbb{D}}

\newcommand{\NN}{\mathbb{N}} %\newcommand{\OWi}{\mbox{ otherwise }}

\newcommand{\PP}{{\mathbb{P}}}

\newcommand{\EE}{ \mathbb{E} }
\newcommand{\TT}{{\rm I \kern -0.2em T}}

\newcommand{\DEQS}{\begin{eqnarray*}}
\newcommand{\EEQS}{\end{eqnarray*}}
\newcommand{\DEQSZ}{\begin{eqnarray}}
\newcommand{\EEQSZ}{\end{eqnarray}}
\newcommand{\DEQ}{\begin{eqnarray}}
\newcommand{\EEQ}{\end{eqnarray}}

\usepackage{color}
\usepackage{amsmath}   % many want amsmath extensions
\usepackage{amsfonts,amssymb, color}
\usepackage{latexsym}
\input colordvi

\usepackage{amsfonts}

\theoremstyle{plain}

\numberwithin{equation}{section}

\begin{document}

\title
[uniqueness  ]{Uniqueness of the nonlinear Schr\"odinger Equation driven by jump processes}
\thanks{The second named author was supported by the FWF-Project P17273-N12.}
\thanks{The research of the third named author on this work was supported by the GA\v CR grant no. GA 15-08819S}

\author{Anne de Bouard, Erika Hausenblas \and Martin Ondrej\'at}

\address{Anne de Bouard, CMAP, Ecole Polytechnique, CNRS, Universit\'e Paris-Saclay, 91128 Palaiseau}
\address{\texttt{debouard{@}cmap.polytechnique.fr}}
\address{Erika Hausenblas, Lehrstuhl Angewandte Mathematik, Montanuniversit\"at Leoben, Franz-Josef-Strasse 18, 8700 Leoben}
\address{\texttt{erika.hausenblas{@}unileoben.ac.at}}
\address{Martin Ondrej\'at, The Institute of Information Theory and Automation, Pod Vod\'arenskou v\v e\v z\'{\i} 4, CZ-182 08, Prague 8, Czech Republic}
\address{\texttt{ondrejat{@}utia.cas.cz}}

\date{\today}

\begin{abstract}
In a recent paper by the first two named authors, existence of martingale solutions to a stochastic nonlinear Schr\"odinger equation driven by a L\'evy noise was proved.
In this paper, we prove pathwise uniqueness, uniqueness in law and existence of strong solutions to this problem using an abstract uniqueness result of Kurtz.
\end{abstract}

\maketitle

%\end{opening}

%{\footnotesize\tableofcontents}

%\title
\newtheorem{convention}{Convention}[section]
\newtheorem{theorem}{Theorem}[section]
\newtheorem{hypo}{Hypothesis}[section]
\newtheorem{notation}{Notation}[section]
\newtheorem{claim}{Claim}[section]
\newtheorem{lemma}[theorem]{Lemma}%[section]
\newtheorem{corollary}[theorem]{Corollary}%[section]
\newtheorem{example}[theorem]{Example}%[section]
\newtheorem{assumption}[theorem]{Assumption}
\newtheorem{tlemma}{Technical Lemma}[section]
\newtheorem{definition}[theorem]{Definition}%[section]
\newtheorem{remark}[theorem]{Remark}%[section]
\newtheorem{hypotheses}{H}
\newtheorem{proposition}[theorem]{Proposition}%[section]
\newtheorem{Notation}{Notation}
\renewcommand{\theNotation}{}

\renewcommand{\labelenumi}{\roman{enumi}.)}

\textbf{Keywords and phrases:} {Uniqueness results, Yamada-Watanabe-Kurtz Theorem, Stochastic integral of jump type,
stochastic partial differential equations, Poisson random measures, L\'evy processes,
Schr\"odinger  Equation.}

\textbf{AMS subject classification (2002):} {Primary 60H15;
Secondary 60G57.}

% \tableofcontents

\section{Introduction}\label{sec_intro}

This paper is a natural continuation of \cite{meandanne}, where the first two named authors proved existence of global solutions to a stochastic non-linear Schr\"odinger equation driven by a time homogeneous Poisson random measure. The existence of solutions was proved in the weak probabilistic sense, i.e. just on one stochastic basis. The aim of the present paper is to give sufficient conditions in order that these solutions are also strong and unique, i.e. global solutions exist on every stochastic basis and are unique pathwise as well as in law. We proceed first by proving pathwise uniqueness of the solutions and then we apply a result of Yamada-Watanabe-Kurtz to show that these solutions are strong and unique in law.

The Yamada-Watanabe theory has been well developed for stochastic equations driven by Wiener processes, see e.g.\ \cite{cherny,engelbert,jacod,martin1,tappe,chinese,yamada} or \cite{mezerdi} for forward-backward stochastic differential equations. There are also analogous results for equations driven by Poisson random measures. For instance, in \cite{poisson1}, the authors develop the Yamada-Watanabe theory for stochastic differential equations driven by both a Wiener process and a Poisson random measure (where the latter is defined on a locally compact space) via the original method of Yamada and Watanabe \cite{yamada}. In \cite{poisson3}, the Yamada-Watanabe theory is presented for variational solutions of partial differential equations driven by a Poisson random measure on a locally compact space also by the method of Yamada and Watanabe \cite{yamada}. Unfortunately, none of these results is applicable to our problem, mainly because the noise does not live in a locally compact space.

Let us remind the reader that the original proof of Yamada and Watanabe is based on an application of a theorem on existence of a regular version of a conditional probability and their idea has proved in time to be so strong and robust to be applicable not only to stochastic differential equations but also to stochastic partial differential equations driven by various noises. Yet, in 2007, Kurtz \cite{kurtz1} presented an abstract Yamada-Watanabe theory aiming not only at stochastic equations but also at many other stochastic problems of different nature. In that paper, Kurtz was the first one to abandon the original idea of the proof of Yamada and Watanabe (regular version of a conditional probability) and based his proof on the universal Skorokhod representation theorem. This approach made it possible to raise the Yamada-Watanabe theory to an abstract level (see also \cite{kurtz2}) where details of particular problems, to which it is applicable, play no role. On the other hand, this abstract approach has one disadvantage for applications that everyone who wishes to apply the result must translate his particular problem to the language of \cite{kurtz1} which itself is not straightforward.

In this paper, we recourse to \cite{kurtz1} due to its generality. We consider mild solutions to stochastic partial differential equations in Banach spaces driven by time homogeneous Poisson random measures on a Polish spaces (which is not in general locally compact, so we cannot apply the results \cite{poisson1} and \cite{poisson3}), we translate this problem to the language  of \cite{kurtz1} and prove the standard existence of unique strong solutions. This result is then applied to the stochastic non-linear Schr\"odinger equation driven by a time homogeneous Poisson random measure.

%some word why it is important
To be more precise,
let $A=\Delta$ be the Laplace operator with $D(A)=\{ u\in L ^2 (\RR ^d): \Delta u \in L ^2 (\RR ^d)\}$.
We are interested in the solution of the following equation
\del{\DEQS
\lqq{ i {\dot{u}(t,x) %\over dt
} +\lk( \Delta u(t,x) +\lambda |u(t,x)|^{\alpha-1}
u(t,x) \rk) } &&
\\ &
=&  u(t,x)\,\dot{L_t}+ u(t,x) z_\nu \, dt ,\quad x\in\RR^d,\, t\in\INT.
\EEQS

Here, $\dot{L}_t$ is a symbol for the  Radon Nikodym derivative of a L\'evy process with characteristic measure $\nu$ and taking values in a Banach function space $Z$, $z_\nu\in Z$, and $\lambda\ge 0$. This equation can be rewritten as}
\DEQSZ\label{itoeqn1}
\\
\nonumber \lk\{ \baray \lqq{ i \, d u(t,x)  -  \Delta u(t,x)\,dt +\lambda|u(t,x)|^{\alpha-1} u(t,x) \, dt\hspace{2cm}}&&
\\&=& \int_S u(t,x)\, g(z(x))\,\tilde \eta (dz,dt)+\int_S u(t,x)\, h(z(x))\, \nu(dz)\, dt,\quad t\in\INT,
\\
u(0)&=& u_0, \earay \rk.
 \EEQSZ
{where $\eta$ denotes the Poisson random measure corresponding to $L$ and $\tilde\eta$ the compensated Poisson random measure, $\nu $ the intensity of the Poisson random measure.}
This equation can be rewritten in terms of a L\'evy process having characteristic measure $\nu$. For more details on the connection of L\'evy processes and Poisson random measure we refer to section 2.3 in \cite{reacdiff} and the references therein. We would like to remark, that Poisson random measures are more general than L\'evy processes.

If the stochastic perturbation is a Wiener process, the equation is well treated and existence and uniqueness of the solution is {known}.
For more information see  \cite{anne2,anne1}. In case the stochastic perturbation is replaced by a \levy process with infinite activity, de Bouard and Hausenblas could only show in \cite{meandanne} the existence of a solution, without uniqueness. Here in this work we are interested {in conditions under which} a unique solution exists.

Since we will use it later on, we will introduce some notations.
\begin{notation}\label{notationref}
$\RR$ denotes the real numbers, $\RR^+:=\{ x\in\RR:x>0\}$ and $\RR^+_0:=\RR^+\cup\{0\}$. By $\mathbb{N}$ we denote the set of natural numbers (including $0$) and by $\bar{\mathbb{N}}$ we denote the set $\mathbb{N}\cup\{\infty\}$.
\end{notation}

\begin{notation}\label{notationref2}
{If $(\mathcal F_t)_{t\in\INT}$ is a filtration and $\theta$ a measure then we denote by $\mathcal F_t^\theta$ the augmentation of $\mathcal F_t$ by the $\theta$-null sets in $\mathcal F_\infty$.}
\end{notation}

\begin{definition}\label{notationref3}
{
A measurable space $(S,\CS)$ is called Polish if there exists a metric $\varrho$ on $S$ such that $(S,\varrho)$ is a complete separable metric space and $\CS=\mathscr B(S,\varrho)$.
}
\end{definition}

\begin{notation}\label{notationref4}
{
The set of all {finite non-negative} measures on a Polish space $(S,\CS)$ will be denoted by $M_+(S)$ and $\CP_1(S)$ will stand for probability measures on $\CS$. If a family of sets $\{S_n\in\CS:n\in\NN\}$ satisfy $S_n\uparrow S$ then $M_{\bar \NN}(\{S_n\})$ denotes the family of all $\bar{\mathbb{N}}$-valued measures $\theta$ on $\CS$ such that $\theta(S_n)<\infty$ for every $n\in\Bbb N$. By $\CM_{\bar \NN}(\{S_n\})$ we denote the $\sigma$-field on $M_{\bar \NN}(\{S_n\})$ generated by the functions $i_B:M_{\bar \NN}(\{S_n\})\ni\mu \mapsto \mu(B)\in \bNN$, $B\in \CS$.
}
\end{notation}

{The proof of the following result shall be deferred to the appendix, see Lemma \ref{lc3} in Section \ref{pomesp}.}

\begin{lemma}\label{measure_lemma}
{
Let $(S,\CS)$ be a Polish space and the family $\{S_n\in\CS\}$ satisfy $S_n\uparrow S$. Then $(M_{\bar \NN}(\{S_n\}),\CM_{\bar \NN}(\{S_n\}))$ is a Polish space.
}
\end{lemma}

\del{\color{blue}
\begin{definition} If $\eta\in\CMM$ then we define
\begin{equation}\label{etat}
\eta_t(V)=\eta(V\cap(S\times(0,t])),\,\eta^t(V)=\eta(V\cap(S\times(t,\infty))),\,V\in\CB(S\times\Bbb R^+).
\end{equation}
\end{definition}
}

\del{\begin{remark}
\rmfy{The intensity measure is $\sigma$--finite, that means one can exclude $\infty$ !}
{\color{RawSienna} A problem here:}
If $X$ is a separable metric space {\color{RawSienna} in fact, we need a Polish space here}, then \blufy{$M_{[0,1]}(X)$} \refy{$\CM(X)$} is metrizable (see Parthasarathy Theorem 6.2, p. 53 {\color{RawSienna} this reference works only for probability measures but measures in $M_{\overline{\Bbb N}}(X)$ can attain infinity.}). Proposition 3.10 Vakhania, page 48 {\color{RawSienna} this reference also works only for probability measures} says that
the topology defined by the sets coincides with the topology defined by open or closed sets.
\Red{!!!!!!!!!!!!!!!} {\color{RawSienna} We need this remark to hold for $M_{\overline{\Bbb N}}(Z)$ and we need it to be not only a separable but also a complete space.}
\end{remark}
}

\section{Time homogeneous Poisson random measures}

Since the definition of time homogeneous Poisson random measure is introduced in many, not always equivalent ways, we give here {our} definition.
\begin{definition}\label{def-Prm}(see \cite{ikeda}, Def. I.8.1)
Let $(S,\CS)$  be a Polish space, %(a Blackwell space)
$\nu$ a $\sigma$--finite measure on $(S,\CS)$, {$\{S_n\in\CS\}$ such that $S_n\uparrow S$ and $\nu(S_n)<\infty$ for every $n\in\Bbb N$}.
A
{\sl time homogenous Poisson random measure} $\eta$ %{on $(S,\CS)$}
over a filtered  probability space $(\Omega,\CF,\BF,\PP)$, where $\BF=(\CF_t)_{t\in\INT}$,  is a measurable function
$$\eta: (\Omega,\CF)\to ({ M_{\bar{\Bbb N}}(\{S_n\times\Int \}),\mathcal M_{\bar{\Bbb N}}(\{S_n\times\Int \})}) ,
$$
such
that
{
\begin{trivlist} \item[(i)]
for each $B\in  \CS \otimes
\mathcal{B}({\mathbb{R}^+}) $ with $\EE \eta(B) < \infty$
 $\eta(B):=i_B\circ \eta : \Omega\to \bar {\mathbb{N}} $ is a Poisson random variable with parameter   $\EE \eta(B)$, otherwise  $\eta(B)=\infty$ a.s.
\item[(ii)] $\eta$ is independently scattered, i.e.\ if the sets $
B_j \in   \CS\otimes \mathcal{B}({\mathbb{R}^+})$, $j=1,\cdots,
n$, are  disjoint,   then the random variables $\eta(B_j)$,
$j=1,\cdots,n $, are mutually independent;
%\item[(iii)] for all $B\in  \CS $ and $I\in \mathcal{B}({\mathbb{R}^+})$, $\mathbb{E}\big[\eta (B\times I)\big]=\lambda(I)\nu(B)$, where $\lambda$ is the Lebesgue measure;
\item[(iii)] for each $U\in \CS$, the $\bar{\mathbb{N}}$-valued
process $(N(t,U))_{t\in \INT }$  defined by
$$N(t,U):= \eta(U \times (0,t]), \;\; t\in \INT $$
is $\BF$-adapted and its increments are stationary and independent of the past,
i.e.\ if $t>s\geq 0$, then $N(t,U)-N(s,U)=\eta(U \times (s,t])$ is
independent of $\mathcal{F}_s$.
\end{trivlist}
}
%{\color{Red} What is the role of $\nu$ in this definition? It does not appear anywhere.}
\end{definition}

%{Ik-Wat-81}

\begin{remark}\label{compensator}
In the framework of Definition \ref{def-Prm}  the assignment
$$\nu: \CS \ni  A \mapsto %)=
 \mathbb{E}\big[\eta( A\times(0,1))\big] %,\; A\in \CS.
 $$
defines a  uniquely determined measure, called in the following {\sl intensity measure}.
%In particular, the intensity measure  of a time homogeneous Poisson random measure is uniquely determined.

\end{remark}

In addition, the term of Poisson random measure is sometimes defined in another way, starting with the intensity measure
and defining the Poisson random measure with given intensity measure.
However, we put the equivalence with the  other definition as a Lemma.

%In fact, if a $\sigma$--finite measure satisfy Assumption \ref{ourass}, then
%the measure induced on $E$ by $\xi$ is a \levy measure.
%Hence, from now we will fix $p\in[1,2]$ and assume during the whole paper that Assumption \ref{ourass} is valid.

\begin{lemma}\label{eqdefprm} \refy{Let $S$ be a separable metric space.}
A measurable mapping $\eta:\Omega\to{M_{\bar{\Bbb N}}(\{S_n\times\Int \})}$ is a time homogeneous Poisson random measure with intensity  $\nu$ iff
\begin{letters}\label{prmpoints}
\item for any $\AAA \in\CZ$ with $\nu(U)<\infty$, the random variable $N (t,\AAA ) $ is
Poisson distributed with parameter $t\,\nu(\AAA )$, otherwise $\PP\lk( N(t,U)=\infty\rk)=1$;
\item for any $n$ and
disjoint sets $\AAA _1,\AAA _2,\ldots,\AAA _n\in\CZ$, and any $t\in[0,T]$, the random
variables $N (t,{\AAA _1}) $, $N (t,{\AAA _2}) $, \ldots, $N (t,{\AAA _n}) $  are  mutually independent;
\item the ${M_{\bar{\Bbb N}}(\{S_n\})}$-valued process  $(N (t,\cdot))_{t\in \Int }$ is adapted to $\BF$; %(\CF_t)_{t\ge 0}$; %In particular, for any $t\in0,T$ and any $\AAA \in\CB(Z)$ the random variable
%$j_{\AAA \times (0,t]}\circ \mu $ is $\CF_t$ measurable;
%
\item for any
$t\in[0,T]$, $\AAA \in\CS$, $\nu(U)<\infty$, and any $r,s\ge t$, the random variables
$N (r,\AAA )-N (s,\AAA )$ are independent of $\CF_t$.
\end{letters}
\end{lemma}

\begin{proof}
To see the  equivalence of (i) and (a), first observe, that  if $S$ is a separable metric space, the Borel space $\CB(S\times \Int )$ of the cartesian product $S\times \RR^ +$ is the product of
the Borel spaces $\CB(S)$ and $\CB(\RR^ +)$, see \cite[p. 6, Theorem 1.10]{parth}. This implies the equivalence from (i) to (a).
To show the equivalence of (ii) and (b), one has in addition to take into account that the Borel $\sigma$--algebra can be generated by intervals of the form $\{(0,t]:t\in\Int \}$.
The equivalence of (iii), and (c) and (d) follows by the definition of $N (t,{\AAA})$.
\end{proof}

Usually, one starts with specifying the measurable space $(S,\CS)$ and the intensity measure $\nu$ on $(S,\CS)$. Given this, then there exists a
Poisson random measure {on $(S,\CS)$ having} the intensity measure $
\nu$.

\medskip

In order to define a stochastic integral with respect to the Poisson random measure,
$S$ has to be related to a topological vector space and the measure $\nu$ has either to be finite or has to be a L\'evy measure.

\begin{definition}\label{def:levy}(See \cite[Chapter 5.4]{linde})
Let $E$ be a separable Banach space with dual $E^\ast$.
A symmetric $\sigma$-finite Borel measure $\nu$ on $E$ is called a {\sl
symmetric L\'evy measure} if and only if
\begin{trivlist}
\item[(i)]
 $\nu(\{0\} )=0$, and
\item[(ii)]
 the
function
$$
E^\ast \ni a\mapsto  \exp \left( \int_E (\cos\langle x,a\rangle
-1) \; \nu(dx)\right)
$$
is the characteristic function of a Radon measure on $E$.
\end{trivlist}
A $\sigma$-finite Borel measure $\nu$ on $E$ is
called a L\'evy measure provided its symmetrisation part $\tilde{\nu}$ is a symmetric L\'evy measure.
\end{definition}
\del{\begin{remark}
As remarked in \cite[Chapter 5.4]{linde} we do not need to suppose
that the integral $\int_E (\cos\langle x,a\rangle -1) \;
\nu(dx)$ is finite. However, see ibid. Corollary 5.4.2, if
$\nu$ is a symmetric L\'evy measure, then, for each $a \in E^\ast$, the integral in question is finite.
\end{remark}
}

\begin{remark}
Let $S$ be a separable Banach space, and $\CS$ its Borel $\sigma$--algebra.
If  the intensity measure
$
\nu:\CS\to\RR
$
{satisfies} the integrability condition
$$
\sup_{a\in S^ \ast\atop |a|\le 1}  \int_{S} 1\wedge |\la z,a\ra |^2\,  \nu(dz)<\infty.
$$
then $\nu$ is a \levy measure (see \cite[Proposition 5.4.1, p.\ 70]{linde}). %for some $p\in[1,2]$.
\end{remark}

For some Banach spaces, one can characterize the \levy measures in a more precise way. Therefore, let us introduce the following definition. Let $\{\ep_k:k\in\NN\}$ be a sequence of independent, identically distributed  random variables with $\PP\lk( \ep_1=1\rk)=\PP\lk( \ep_1=-1\rk)=\frac12$. Then a Banach space with norm $|\cdot|$ is of $R$--type $p$ (Rademacher type $p$), if for any sequence $\{ x_j:j\in \NN\}$ belonging to $l_p(E)$, we have (compare \cite[p.\ 40]{linde})%the element
$$\PP\lk(\Big| \sum_{j=1}^\infty \ep_j x_j\Big|<\infty\rk)=1
.
$$
The Minkowski inequality implies, that each Banach space is of $R$--type $1$, the range of $p$ is usually between one and two.

\begin{remark}\label{ourass1}
Let $(S,\CS)$ be a Polish space, the family $\{S_n\in\CS\}$ satisfy $S_n\uparrow S$, and $\nu$ be a $\sigma$--finite measure with $\nu(S_n)<\infty$ for any $n\in\NN$.
Fix $p\in[1,2]$. We assume that $E$ {is} a separable Banach space of $R$--type $p$, and that
$\xi:(S,\CS)\to (E,\CB(E))$ {is} a  measurable mapping.
In addition, we assume that the intensity measure
$
\nu:\CS\to\RR_0^+
$
{satisfies} the integrability condition
\DEQSZ\label{heregrow}
  \int_{S} 1\wedge | \xi(z) |_E^p\,  \nu(dz)<\infty,\quad  \mbox{and} \quad  \nu(\{0\})=0.
\EEQSZ
Then, the measure $\nu_E$ induced by $\xi$ on $E$ is a \levy measure (and $\nu_E(\{0\}):=0$) (compare \cite[p.\ 75]{linde}). In addition, if $\eta$ is a Poisson random measure with intensity $\nu$
over a filtered probability space $(\Omega,\CF,\BF,\PP)$,  the process
$$
L:\INT %[0,\infty)
\ni t \mapsto \int_0^ t \int_S \xi(z\,)\, ( \eta-\nu\times \lambda)(dz,ds)
$$
is a \levy process over  $(\Omega,\CF,\BF,\PP)$.
\end{remark}

Hence, from now on we will %fix $p\in[1,2]$ and
assume during the whole paper that the following convention %\ref{ourass}
 is valid.

\begin{convention}\label{ourass}
{We convene that} $(S,\CS)$ {is} a Polish space, %(a Blackwell space)
$\nu$ a $\sigma$--finite measure on $(S,\CS)$ and $S_n\in\CS$ such that $S_n\uparrow S$ and $\nu(S_n)<\infty$ for every $n\in\Bbb N$.
\end{convention}

%
%\end{document}
%
%
%
%
%%If $(S,\CS)$ is a measurable space
%%

Let us consider a  filtered  probability space $(\Omega,\CF,\BF,\PP)$, where $\BF=\{\mathcal{F}_t\}_{t\in \INT }$ denotes a
filtration.
%Let $X$ be any Banach space.
%Let $\MA=(\Omega,\CF,\BF,\PP)$  be a complete filtered probability space. Here, we assume that $\BF$ is right continuous.
A process $\xi:[0,T]\times \Omega\to X$
is progressively measurable, or simply, progressive, if its restriction to $\Omega\times [0,t]$ is
$\CF_t\otimes\CB([0,t])$--measurable for any $t\ge 0$.
The predictable random field $\CP $ on $\Omega\times
\RR _ +$ is the $\sigma$--field generated by all continuous
$\BF$--adapted  processes (see e.g.\ Kallenberg \cite[Chapter
25, p.\ 491]{kallenberg}).

A real valued stochastic process $\{x(t):t\in\INT \}$,  defined on a
filtered probability space $(\Omega,\CF,\BF,\PP)$ is  called
predictable, if the mapping  $x:\Omega\times\Int \to\Rb{}$ is
$\CP/\mathcal{B}(\RR)$-measurable. A random measure $\gamma$ on {$\mathcal S\otimes\mathcal B(\Int )$}
over $(\Omega;\CF,\BF,\PP)$ is called
predictable, iff for each $U\in \CS$, the $\RR$--valued process
$\Int \ni t\mapsto \gamma( U\times(0,t])$ is predictable.

\begin{definition}\label{def-imPrm}
Assume that  $(S,\CS )$ is a measurable space and $\nu$ is a
non-negative $\sigma$--finite measure on $(S,\CS )$. Assume that $\eta$ is a   {time homogeneous} Poisson
random measure  with  intensity measure $\nu$ on  $(S,\CS )$
over $(\Omega,\CF,\BF,\PP)$.  The {\sl compensator} of  $\eta$ is
the unique {predictable random measure}, denoted by $\gamma$, on {$\mathcal S\otimes\mathcal B(\Int )$} over {$(\Omega,\CF,\BF,\PP)$}, such that
for each $T<\infty$ and $A\in \CZ $ with $\EE\eta( A\times
(0,T])<\infty$, the $\mathbb{R}$-valued processes
$\{\tilde{N}(t,A)\}_{t\in (0,T]}$  defined by
$$
\tilde{N}(t,A):= \eta(  A\times (0,t] )-\gamma( A\times (0,t] ),
\quad   0< t\le T,
%t \in \mathbb{R}^+
$$
is a martingale on $(\Omega,\CF,\BF,\PP)$.
\end{definition}
\begin{remark}
Assume that $\eta$ is a time homogeneous Poisson random measure
with intensity $\nu$ on  $(S,\CS)$ over $(\Omega,\CF,\BF,\PP)$. It
turns out that the compensator $\gamma$ of $\eta$ is uniquely
determined and moreover
$$
\gamma: \CS \times \CB(\RR^+)\ni (A,I)\mapsto  \nu(A)\times
\lambda(I).
$$
Here $\lambda$ denotes the Lebesgue measure on $\RR$.
The  difference between a time homogeneous  Poisson random measure
$\eta$  and its compensator $\gamma$, i.e.  $\tilde
\eta=\eta-\gamma$, is called a  {\sl compensated Poisson random
measure}.
\end{remark}

Let  $(S,\CS)$ be a measurable space
and
let $\eta$ be a time homogenous Poisson random measure on $S$ with intensity  measure $\nu$ being a positive $\sigma$--finite measure % \in {L}(Z)$
 over
$\mathfrak{A}$ satisfying Convention \ref{ourass}.
We will denote by $\tilde{\eta }$ the \it  compensated Poisson random measure \rm  defined by
$\tilde{\eta }: = \eta - \gamma $, where the compensator
 $\gamma : \bcal (\Int )\times \CS  \to {\Int }$ satisfies in our case the following equality
$$
     \gamma (I \times B) = \lambda (I) \nu (B) ,
\qquad I \in \bcal ({\Int }) , \quad  B \in \CS .
$$

\begin{lemma}\label{poiss_uni_distr}
Let \refy{$(S,\CS)$ be a measurable space and} $\nu$ be a non--negative $\sigma$--finite  measure on $S$  satisfying Convention \ref{ourass}. % and $Z$ be a separable metric space.
Then the following holds
\begin{enumerate}
  \item there exists a probability space $\mathfrak{A}=(\Omega,\CF,\PP)$ and a time homogenous Poisson random measure $\eta:\Omega\to {M_{\bar{\Bbb N}}(\{S_n\times
  \Int \})}$ with the intensity measure $\nu$;
  \item \refy{In addition, let us assume that $S$  is  a separable metric space.} Denote by $\Theta_\nu$ the law of $\eta$ on ${\mathcal M_{\bar{\Bbb N}}(\{S_n\times\Int \})}$.  If $\eta^\sharp$ is a time homogenous Poisson random measure defined possibly on different stochastic base $\mathfrak{A}^\sharp=(\Omega^\sharp,\CF^\sharp,\PP^\sharp)$ and $\nu$ is the intensity measure for $\eta^\sharp$ then $\Theta_\nu$ is the law of $\eta^\sharp$ on ${\mathcal M_{\bar{\Bbb N}}(\{S_n\times\Int \})}$.
\end{enumerate}
\end{lemma}

\begin{proof}
Part i.) is given by Theorem 8.1 \cite[p.\ 42]{ikeda}. It remains to show ii.).
Since $\nu$ is $\sigma$--finite, there exists a increasing family $\{S_n:n\in\NN\}$ with $S_{n+1} \supseteq S_n$, $S_n\uparrow S$, and $\nu(S_n)<\infty$.
To show that $\eta$ and $\eta^ \sharp$ have the same law on $\CM_{\bar \NN}(S\times\Int )$, we have to show that for all $f:S\times \Int \to\RR$,
bounded and continuous, the random variable $\eta(f):= \int_{S_n}\int_0^ T  f(s,t)\,\eta(ds,dt)$ and $\eta^ \sharp(f):= \int_{S_n}\int_0^ T f(s,t)\, \eta^ \sharp(ds,dz)$ have the same law, see \cite[Theorem 5.8, p.\ 38]{parth}.
Since $S\times \RR^ +$ is a Polish space,  the $\sigma$ algebra generated by the family of bounded continuous functions coincides with the Borel--$\sigma$--algebra, see \cite[Proposition 1.4, p.5]{vakhania}. Therefore, it is sufficient to show for all $n\in\NN$, $U\in\CB(S_n)$ and $I\in\CB(\Int )$, that the random variables
$\eta(U\times I)$ and $\eta^ \sharp(U\times I)$ have the same law.
Let $\Theta_\nu^\sharp$ be the law of $\eta_\sharp$ and let us assume $\nu(U),\lambda(I)<\infty$.
Let $k\in\NN_0$. Then, by the definition of the Poisson random measure and its intensity measure $\nu$ we know that
\DEQS
\Theta_\nu( \eta(U\times I])=k) &= & e^{-\lambda(I)\nu(U)}\, {(\nu(U)\,\lambda(I))^k\over k!}
\\
&=&\Theta_\nu^\sharp( \eta_\sharp(U\times I)=k).
\EEQS
If $\nu(U)=\infty$ or $\lambda(I)=\infty$, then $\Theta_\nu\lk( \eta(U\times I)=\infty\rk)=1=\Theta_\nu^ \sharp\lk( \eta(U\times I)=\infty\rk)$.
\end{proof}

Now, one can define the stochastic integral with respect to the Poisson random measure.
Here, one has two possibilities at ones disposal, to use predictable integrands, or more general, to use progressively measurable integrands.  The stochastic integral with predictable integrands is introduced e.g.\ in the book of Ikeda and Watanabe \cite{ikeda} or in the book of Applebaum \cite{apple}, the stochastic integral with progressively measurable integrands is introduced e.g.\  in
  \cite{Brz+Haus_2009} in $M$--type $p$ Banach spaces.

\begin{definition}\label{def-mart-intext}
Let $0<p\le 2$. A Banach space $E$ is of  martingale type $p$  iff
there exists a constant $C>0$ %L_{p}(E)>0$
such that for all
$E$-valued  finite martingale $\{M_{n}\}_{n=0}^N$  the
following inequality holds
\DEQSZ\label{nnn} \sup_{0\le n\le N } \mathbb{E} | M_{n} |_E ^{p} \le  C\, % L_{p}(E)\,
\mathbb{E}  \sum_{n=0}^N  | M_{n}-M_{n-1} |_E ^{p},
\EEQSZ %\label{eqn-2.1}\end{equation}
 where as  usually, we put  $M_{-1}=0$.
% We denote the smallest constant satisfying \eqref{nnn} by $C_p(E)$.
\end{definition}

Examples of $M$--type $p$ Banach spaces are, e.g.\ $L^q(\CO)$ spaces, where $\CO$ is a bounded domain. $L^q(\CO)$ is of $M$--type $p$ for any $p\le q$ (see e.g.\ \cite[Chapter 2, Example 2.2]{mtypep}).
Is a Banach space $E$ is of $M$--type $p$ and $A$ a generater of an analytic semigroup on $E$, then the complex interpolation spaces between $D(A)$ and $E$ are of $M$--type $p$.
Similar fact holds also for real interpolation spaces, but not in this generality, for more details we refer to Appendix A of \cite{zmtypep}.
In particular, in   \cite{Brz+Haus_2009}  it  is proven that for any Banach space $E$ of $M$--type $p$ there exists a unique
continuous linear operator $I$ which associates   to  each progressively measurable
process $\xi: {\mathbb{R}}_{+} \times \Omega \to L^p(S,\nu;E) $  with $\PP$-a.s.\
\begin{equation} \label{cond-2.01}
   \int_{0}^{T} \int_{S} \vert \xi (r,x) \vert_E^{p} \, \nu (dx) dr   < \infty,
\end{equation}
for every $T>0$, an adapted $E$-valued c\`{a}dl\`{a}g process
$$
  {I}_{\xi , \tilde{\eta }} (t):=\int_{0}^{t} \int_{S} \xi (r,x) \tilde{\eta } (dr,dx), \quad t \ge 0
$$
such that if a process $\xi$ satisfying  the above condition (\ref{cond-2.01}) is a
random step process with representation
\DEQSZ\label{kkk}
\xi(r,x) = \sum_{j=1} ^n 1_{(t_{j-1}, t_{j}]}(r) \, \xi_j(x),\quad x\in S,\quad  r\in \INT ,
\EEQSZ
where $\{t_0=0<t_1<\ldots<t_n<\infty\}$ is a finite partition of $[0,\infty)$ and
for all $j\in \{ 1, \ldots ,n \} $,    $\xi_j$ is  an $E$-valued $\CF_{t_{j-1}}$--measurable $p$-summable simple  random variable,
then
\begin{equation} \label{eqn-2.02}
 I_{\xi,\tilde{\eta }} (t) % \int_0^t \int_S \xi(r,x)\tilde{\eta}(dx,dr)
=\sum_{j=1}^n  \int_S  \xi_j (x) \,\tilde \eta \lk((t_{j-1}\wedge t, t_{j}\wedge t] ,dx \rk),\quad   t\in \INT .
\end{equation}
This definition can be extended to all progressively measurable mappings $\xi:\Omega\times \INT \times S\to E$  with $\PP$--a.s.
\DEQS
\int_0^ T \int_S\min(1,|\xi(r,z)|^p_E) \nu(dz)\, dr<\infty .
\EEQS
Some information of the different setting is given in \cite{ruediger}.

In addition, we would like to point out in the following Proposition,  that we do not need to suppose that the filtration of the given probability space is right continuous.
In particular,  given a Poisson random measure $\eta$ over a filtered probability space $(\Omega,\PP,\CF,\BF)$, $\BF=(\CF_t)_{t\in \INT}$, with an arbitrary filtration,
 a progressively measurable  $L^2(S,\nu)$--valued process $\xi$, one can pass to the right continuous augmentation of the filtration without loosing the necessary  properties.
 In particular, the following holds.
% In fact, this is given by Theorem 7.23 of \cite{kallenberg}.

\begin{proposition}\label{ehpro}
Let $E$ be a Banach space $E$ of $M$--type $p$, $(S,\CS)$ a measurable space {subject to} Convention  \ref{ourass} and $\eta$ a Poisson random measure
on a filtered probability space $\mathfrak{A}=(\Omega,\PP,\CF,\BF)$ {with the intensity measure $\nu$}, with an arbitrary filtration $\BF=(\CF_t)_{t\in \INT }$.
Let  $\xi: \INT \times \Omega \to L^p(S,\nu;E) $  be a
 progressively measurable
process with $\PP$-a.s.\
\begin{equation} \label{cond-2.01.1}
   \int_{0}^{T} \int_{S} \vert \xi (r,x) \vert_E^{p} \, \nu (dx) dr   < \infty.
\end{equation}
Let $\bar {\mathfrak{A}}=(\Omega,\PP,\CF,\tilde \BF)$, $\tilde \BF=(\bar \CF_t)_{t\in \INT }$ be the right continuous filtration given by $\bar \CF_t:= \wedge _{h>0}\mathcal F^{\Bbb P}_{t+h}$
and let  $\bar \xi: {\mathbb{R}}_{+} \times \Omega \to L^p(S,\nu;E) $  be
an $\bar \BF$--progressively measurable
process with $\PP$-a.s.\
\begin{equation} \label{cond-2.01.2}
   \int_{0}^{T} \int_{S} \vert \bar \xi (r,x) \vert_E^{p} \, \nu (dx) dr   < \infty.
\end{equation}
Let us assume that $\xi$ and $\bar \xi$ have the same law on $L^ p([0,T];L^p(S,\nu;E))$.
Let  $\bar\CI$ and $\CI$ be defined by
%Then, for any $t>0$ the random variable
$$
  \overline {\CI} (t):=\int_{0}^{t} \int_{S} \bar\xi (r,x) \tilde{\eta } (dr,dx), \quad t \in \INT,
  $$
and
$$
 {\CI} (t):=\int_{0}^{t} \int_{S} \xi (r,x) \tilde{\eta } (dr,dx), \quad t \in \INT,
$$
where the stochastic integral is defined as before.
Then, the triplets $(\eta,\bar \xi,\bar \CI)$ and $(\eta,\xi,\CI)$ have the same distribution on $\CM(\{S_n\times [0,T]\})\times L^ p([0,T];L^p(S,\nu;E))\times E$.
\end{proposition}

\del{
{\color{Red} Wouldn't this be enough? Instead of Proposition \ref{ehpro}?
\begin{remark} In the setting above, $\eta$ is a Poisson random measure also for the the augmented right continuous filtration $\bar \CF_t:= \wedge _{h>0}\mathcal F^{\Bbb P}_{t+h}$, hence we can construct the stochastic integral on $(\Omega,\mathcal F,(\mathcal F_t),\Bbb P)$, resp. $(\Omega,\mathcal F,(\bar\CF_t),\Bbb P)$
$$
I_1=(\mathcal F_t)-\int_0^t\int_S\xi(r,x)\tilde\eta(dr,dx),\quad\textrm{resp.}\quad I_2=(\bar\CF_t)-\int_0^t\int_S\xi(r,x)\tilde\eta(dr,dx).
$$
In this case, $I_1=I_2$.
\end{remark}
}
}

\begin{proof}
In fact, this is given by Theorem 7.23 of \cite{kallenberg}.
To be more precise, let
$\xi$ be given with representation \eqref{kkk}, then the stochastic integral is defined by
 the martingale $(M_n)_{n=1}^k$, where
  $$
 M_n=\sum_{j=1}^n  \int_S  \xi_j (x) \,\tilde \eta \lk((t_{j-1}, t_{j}] ,dx \rk)
 $$
Since for each $j=1,\ldots, k$, we have $\PP$-a.s.\ for
the conditional expectation
$$
\lim_{n\to \infty } \EE\lk[   \,\tilde \eta \lk((t_{j-1}, t_{j}] \times U  \rk)\mid \CF_{t_{j-1}+\frac 1n }\rk]
=\EE\lk[  \,\tilde \eta \lk((t_{j-1}, t_{j}]\times U  \rk)\mid \CF_{t_{j-1}}^+\rk]
$$
and on the other hand we have
$$
\lim_{n\to \infty } \EE\lk[   \,\tilde \eta \lk((t_{j-1}, t_{j}] \times U  \rk)\mid \CF_{t_{j-1}+\frac 1n }\rk]=
\lim_{n\to \infty } \eta\lk((t_{j-1}, t_{j-1}+\frac 1n ] \times U   \rk)-\nu(U)\, \frac 1n.
$$
Since
$$
%\PP\lk(
 \eta\lk((t_{j-1}, t_{j-1}+\frac 1n ] \times U   \rk)  %>0\rk) \le \nu(U)\, \frac 1n\longrightarrow 0,
$$
is a Poisson distributed random variable with it parameter $\nu(U)\, \frac 1n$, it follows that $\PP$-a.s.\
$\eta\lk((t_{j-1}, t_{j-1}+\frac 1n ] \times U   \rk)\longrightarrow 0$ as $n\to\infty$.
The assertion follows from the definition of the integral.
\end{proof}
%follows the assertion.

%\end{document}
%

\section{{Pathwise uniqueness of the stochastic Schr\"odinger Equation}}

We are interested in uniqueness of the stochastic Schr\"odinger equation driven by a L\'evy noise.
The nonlinear Schr\"odinger equation is an example of a universal nonlinear model that
describes many physical nonlinear systems. The equation can be applied to hydrodynamics,
nonlinear optics, nonlinear acoustics, quantum condensates, heat pulses in solids and various
other nonlinear instability phenomena.
The Schr\"odinger equation arises also in the context of water waves.
In $1968$ V.E.\ Zakharov   derived the Nonlinear Schr\"odinger equation for the two-dimensional water wave problem in the absence of surface tension, that is,
for the evolution of gravity driven surface water waves.
More recently, Villarroel, et all \cite{villa2,villa1}  considered the nonlinear Schr\"dinger equation with
randomly distributed, but isolated jumps. Dealing with jumps one may model sudden changes in the field that occur randomly.

In order to model abrupt changes in the medium, as it can e.g. be the case for the propagation of light in optical fibers,
or of other parameters, one can use L\'evy noise. In \cite{meandanne}
the first and second author investigated the existence of solution, if the L\'evy measure has infinite activity.  However,
no uniqueness was proven.

From now on,
 $S$ will be a {Borel subset of a separable Banach} function space  continuously embedded in the Sobolev space $ W^ 1_{\infty}(\RR^d)$.
 %, such that there exists a constant
 As mentioned in the introduction, the equation we are considering is given by
\DEQSZ\label{itoeqn}
\\
\nonumber \lk\{ \baray \lqq{ i \, d u(t,x)  -  \Delta u(t,x)\,dt +\lambda|u(t,x)|^{\alpha-1} u(t,x) \, dt\hspace{2cm}}&&
\\&=& \int_S u(t,x)\, g(z(x))\,\tilde \eta (dz,dt)+\int_S u(t,x)\, h(z(x))\, \gamma (dz, dt),\quad t\in [0,T],\\
u(0)&=& u_0. \earay \rk.
 \EEQSZ
with $\lambda\ge 0$. Here,  $g:\RR\to \mathbb{C}$ and $h:\RR \to \mathbb{C}$ are  two functions satisfying the following items:
\renewcommand{\theenumi}{(\Roman{enumi})}
\begin{steps}
%\setcounter{enumi}{1}
%\stepcounter{NameOfTheNewCounter}
\item
$g$, $\nabla g$, $h $ and $\nabla h $ are  of linear growth, i.e.\ there exist some constants $C_g$ and $C_h $ such that
$$|g(\xi )|, |\nabla g(\xi )|\le C_g |\xi | \quad \mbox {and} \quad  |h (\xi )|,|\nabla h (\xi )|\le C_h  |\xi |.
$$
  \item $g(0)=0$ and $h (0)=0$.
  \item $|\Im(h (\xi ))|\lesssim |\xi |^ 2$ and $|\Im(g(\xi ))|\lesssim |\xi |^ 2$
\end{steps}
Here, $\Im$ denotes the imaginary part of a number. From condition (i) one can derive from condition \eqref{heregrow}
a similar condition for the associated Nemityski operator defined later on.

\medskip

Let us denote by $(\CT(t))_{t\ge 0}$  the group of isometries generated by the operator $-iA$.
In particular, for any $t\in\RR$ and $u_0\in L^ 2(\RR^ d )$  let us denote the solution of the following
Cauchy problem
\DEQS
\lk\{ \barray i\,\dot u(t) &=& Au(t),\\
u(0)&=& u_0,
\earray\rk.
\EEQS
by  $\CT(t)u_0$.
Observe, $(\CT(t))_{t\ge 0}$  forms a unitary group on $L^ 2(\RR^ d )$ and $H^\gamma_2(\RR^d)$ for any $\gamma\in\RR$.
In the framework of evolution equation one considers the
mild solution of Equation \eqref{itoeqn}, which is given by the following integral equation for $t\in\INT$
% Equation \ref{strat}
% is equivalent in the setting of the semigroup theory to
%
\DEQS
%\label{eq_1}
%\label{eq_1} \tag{SEE}
 u(t)  & = & \CT(t) u_0+ i \lambda \int_0^ t \CT(t-s) (
|u(s)|^{\alpha-1} u(s))\, ds
\\ &&   {}- i \int_0^t\int_S \CT(t-s) \lk[ u(s)\, G(z)\rk] \tilde \eta(dz,ds)
-i \int_0^ t \int_S \CT(t-s) \lk[ u(s) H(z) \rk] \, \nu(dz) \, ds
.
\EEQS
Here, we used the Nemytskii operators  corresponding  to $g$ and $h$. In particular, the mappings $G:L ^ 2(\RR^ d )\to L^ 2 (\RR^ d )$ and $H:L ^ 2(\RR^ d )\to L^ 2 (\RR^ d )$ denote the Nemytskii operators associated to the functions $g$ and $h$ defined by
$$
(G(y)) (x) :=g(y(x)),\quad %z\in Z,\,x\in\RR^ d ,
\mbox{and}\quad  (H(y)) (x) :=h(y(x)),\quad y\in S,\,x\in\RR^ d .
$$

\begin{definition}
Let %$E$ be a Banach space and
$T>0$.
We call $u$ an $L^2 (\RR^d)$--valued mild solution to Equation \eqref{itoeqn} on the time interval $[0,T]$, iff $u$ is an adapted \cadlag process in $L^2 (\RR^d)$,
the terms
$$
\int_0^ t\lk| \CT(t-s)
{[|u(s)|^{\alpha-1} u(s)]}\rk|_{L^2 }\, ds,\quad  %\mbox{and} \quad
 \int_0^t\int_{\{y\in S: |\CT(t-s)u(s)\, G(y)|_{L^ 2}< 1\}} \lk| \CT(t-s) {[u(s)\, G(y)]}\rk|_{L^2 }^ 2 \nu(dy) \, ds,
$$
$$
 \int_0^t\int_{\{y\in S: |\CT(t-s)u(s)\, G(y)|_{L^ 2}\ge 1\}} \lk| \CT(t-s) {[u(s)\, G(y)]}\rk|_{L^2 } \nu(dy) \, ds,
$$
%for some $p\in[1,2]$
 and
$$
 \int_0^t\int_S\lk| \CT(t-s) {[ u(s)\, H(y)]}\rk|_{L^2 }  \nu(dy) \, ds,
$$
are $\PP$-a.s.\ finite  and for any $t\in[0,T]$, the process  $u$ solves $\PP$-a.s.\ the integral equation
\DEQSZ
\label{eq_1}
\\
\nonumber
\lqq{ u(t)   =   \CT(t) u_0+ i\lambda \int_0^ t \CT(t-s)
{[|u(s)|^{\alpha-1} u(s)]}\, ds
  {}}
  &&
  \\\nonumber
  &&{}
  - i \int_0^t\int_S  \CT(t-s) \lk[ u(s) G(y)\rk] \tilde \eta (dy,ds)
-i \int_0^ t \int_S \CT(t-s)\lk[ u(s) H(y)\rk]\nu(dy)\, ds
.
\EEQSZ
\end{definition}

Since in the proof of existence of solution compactness arguments are used, the underlying probability space {gets}
lost. Hence, a concept of probabilistic weak solutions  has to be introduced, which is done in the following definition.

\begin{definition}\label{Def:mart-sol}
%Let $E$ be a  separable  Banach space.
%Let $(S,\CS)$ be a measurable space and $\nu$ a $\sigma$--finite measure on $(Z,\CZ)$.
%Suppose that  $G$ and $H$ are a densely defined function from $ S$
%to $L^2(\RR^d)$.
Let $u_0\in L^2(\RR^d)$.
A {\sl martingale solution} on $L^2(\RR^d)$  to the Problem  \eqref{itoeqn}  is a system
\begin{equation}
\left(\Omega ,{\CF},{\Bbb P},\BF, %\{{\CF}_t\}_{t\ge 0},
\eta,u\right)
\label{mart-system}
\end{equation}
such that
\begin{numlistn}
\item  $(\Omega ,{\CF},\BF,{\Bbb P})$ is a  filtered
probability
space with filtration $\BF=(\CF_t)_{t\in\INT }$,
\item
$\eta$  is a time homogeneous Poisson
random measure on $(S,\CS)$ over $(\Omega
,{\CF},\BF,{\Bbb P})$ with intensity measure $\nu$  satisfying Convention \ref{ourass},
\item $u$  is {an} $L^2 (\RR^d )$--valued mild solution  to  the Problem \eqref{itoeqn}.
\end{numlistn}
% and $u$ is $p$-integrable in $X$.
%
\end{definition}

Let us remind that
 $S$ is a function space  continuously embedded in the Sobolev space $ W^ 1_{\infty}(\RR^d)$. %, such that there exists a constant $C>0$ with $|z|_Z\le C
In addition, we assume that the intensity  measure satisfies the following integrability conditions:
\label{integrabilitycondition}
  \begin{enumerate}
    \item $C_0(\nu):=\int_S |z|_{L^\infty}^ 2 \nu(dz)<\infty;\phantom{\Big|}$
    \item $C_1(\nu):=\int_S |z|_{W^ 1 _\infty}^ 2 \nu(dz)<\infty;\phantom{\Big|}$
    \item  $C_2(\nu):=\int_S %\sup_{x\in {\RR^d}}
 \int_{\RR ^d}    |x|^ 2 |z(x)|^ 2 \, dx\, \nu(dz)<\infty;\phantom{\Big|}$
    \item  $C_3(\nu):=\int_S |z|_{L^ \infty}^ 4 \nu(dz)<\infty.\phantom{\Big|}$
  \end{enumerate}
\begin{remark}
One can  see from the proof of Theorem 2.7 in \cite{meandanne} that the large jumps
have no effect on the uniqueness result and one can easily generalize our Theorem \ref{main} to the case where one has large jumps without any bounded moments.
\end{remark}

Let
$$1\le \alpha< \bcase  1+4/(d-2)& \mbox{ for } d>2,
\\ \infty& \mbox{ for }  d=1 \mbox{ or } 2,\ecase
$$
In \cite{meandanne} we have shown that under the conditions stated above, there exists a {martingale} solution. For the sake of completeness, we state here
the main result of the article. Before, since we will need it later on, let us introduce  the mass by
\DEQSZ\label{mass}
\mathcal{E}(u):= \int_{\RR^d} | u(x)|^ 2 \, dx,
\EEQSZ
and the energy by %Hamiltonian
\DEQSZ\label{hamiltonian}
\\
\nonumber
\mathcal{H}(u):= \frac12 \int_{\RR^d} |\nabla u(x)|^ 2 \, dx + {\lambda\over \alpha+1} %%
\int_{\RR^d } |u(x)|^ {\alpha-1} { u(x)\, \coc{u(x)}\,} dx.
\EEQSZ

%However, since we used compactness arguments, we could not show the uniquness of the solution to \eqref{eq_1}.

\begin{theorem}\label{main}
Let  $\eta$ be  a time homogenous Poisson random measure on  $S$ with \levy measure $\nu$ satisfying
the integrability conditions (i), (ii) and (iii) given  above.
If  $u_0\in H ^1_2(\RR ^d)$, $\lambda\ge 0$, and
$$ \int_{\RR^d} |x|  ^2 |u_0(x)| ^2 \, dx <\infty,
$$
then there exists a $H^1_2(\RR^d)$--valued martingale solution to \eqref{itoeqn}, which is \cadlag in $H^{\gamma}_2(\RR^d)$
for any $\gamma<1$.

In addition,
there exists a constant $C=C(T,C_0(\nu),C_3(\nu),C_g,C_h)$ such that
$$
\EE\sup_{0\le t\le T} |u(t)|_{L^2 }^2 \le  C\, (1+\EE |u(0)|_{L^2 }^2)
$$
and
for any $T>0$ there exists a constant  $C=C(T,C_0(\nu),C_1(\nu), C_g,C_h)>0$ such that
$$
\EE {\sup_{t\in \INT\cap\mathbb{Q}}  \CH(u(t))} \le  C\, (1+ \EE \CH(u(0))).
$$
%{\color{Red}(are you sure the uncountable supremum is measurable?)}
\end{theorem}

The proof of Theorem \ref{main} uses compactness arguments, hence, as mentioned before, the existence of a solution is shown,  but no uniqueness of the solution.
%Let us assume that existence of a solution to Equation \eqref{itoeqn} (in strong or weak sense) is shown.
Here, we are interested in the uniqueness of the solution to Equation \eqref{itoeqn}.
However, similar to the concept of solutions, there exist several concepts of uniqueness.

\begin{definition}
The equation \eqref{itoeqn} is {\sl pathwise unique} if, whenever $(\Omega,\CF,(\CF_t)_{t\in \INT}, \PP, \eta, u_i)$, $i=1,2$ are solutions to \eqref{eq_1} such that $\PP\lk( u _1 (0)=u _2 (0)\rk)=1$, then $\PP\lk( u _1 (t)=u _2 (t)\rk)=1$ for every $0\le t\le T$.
\end{definition}

Under certain conditions pathwise  uniqueness of the stochastic Schr\"odinger equation driven
by L\'evy noise can be shown.

\begin{theorem}\label{main_ex_u}
Let us assume that \del{\DEQSZ\label{assonZ}
 C_0(\nu)<\infty
 \EEQSZ and} $g,h:\RR\to\mathbb{C}$ are Lipschitz continuous.
Let us assume that
\DEQSZ\label{condalpha2}
  \bcase 1\le \alpha<\infty   \quad& \mbox{if}  \quad  d=1,\mbox{ or } 2,\\
 1\le \alpha<\tfrac{d}{d-2}& \mbox{if}\quad  d>2.
\ecase
\EEQSZ

Let be given a filtered probability space $\mathfrak{A}=(\Omega,\CF,\BF, \PP)$, with  filtration $\BF=(\CF_t)_{t\in \INT}$, a Poisson random measure defined on $\mathfrak{A}$
adapted to the filtration $\BF$, and
two mild solutions $u_1$ and $u_2$ to equation \eqref{itoeqn} over $\mathfrak{A}$,
on $[0,T]$ such that %in case $d=1$, or $d=2$
{$u_1$ and $u_2$ are c\`adl\`ag in $L^2(\RR^d)$.
If there exists some $\delta\in\RR$ with
\DEQSZ\label{givendelta}
\delta >\bcase \frac d2-\frac d{2(\alpha-1)}, & \mbox{ if } d=1,2,\\
 \frac d2-\frac 1 {\alpha-1} , & \mbox{ if } d>2,
\ecase
\EEQSZ
and the solutions $u_1$ and $u_2$ are in $\DD([0,T];H^{\delta}_2(\RR^d ))$,}
then $u_1$ and $u_2$ are indistinguishable in $L^2(\RR^d)$. %pathwise uniqueness for the solution to \eqref{eq_1} holds.
\end{theorem}
\begin{remark}
Condition  \eqref{condalpha2} is needed to apply the Strichartz estimate and Sobolev embedding theorems, in order to handle
the nonlinearity. The restriction compared to usual conditions is due to the fact that we have to
estimate the difference of solutions in $L^2$, in order to tackle the stochastic terms.
\end{remark}

\begin{proof}[Proof of Theorem \ref{main_ex_u}]

{Note that the solutions given by Theorem  1.2 of \cite{meandanne} satisfy the assumption of the Theorem \ref{main_ex_u}.
Moreover, if $u_1$ and $u_2$ are as above, then they belong a.s. to
$L^{\infty}([0,T]; H_2^{\delta}(\RR^d))$.}
\del{ In particular, there exists a constant $C=C(T)>0$ such that
\DEQSZ\label{finiteinlinfty}
\EE \sup_{0\le t\le T}|u_i(t)|^2_{H^1_2} &\le & C.
\EEQSZ}
%In addition, both solution processes are  \cadlag on $[0,\tau_m\wedge T]$.
%
%

\medskip

In the first step we will introduce a family of stopping times $\{\tau_m:m\in\NN\}$
and show that on the time interval $[0,\tau_m]$ the solutions $u_1$ and $u_2$ are indistinguishable.
In the second step, we will show that $\PP\lk( \tau_m<T\rk)\to 0$  for $m\to\infty$. From this follows that $u_1$ and $u_2$ are indistinguishable on the time interval $[0,T]$.

\medskip

\paragraph{\bf Step I:}
%Let $\delta<1$ be a constant such that for $d=1,2$, $\delta>\tfrac {d(\alpha-3)}2$
%and for $d\ge 3$, $\delta> \tfrac{d(\alpha-1)-2}{2(\alpha-1)}$. Due to the conditions on $\alpha$ such a $\delta>0$ exists.
Let us introduce the stopping times  $\{ \tau^1_m:m\in\NN\}$ and  $\{ \tau^2_m:m\in\NN\}$  given by
%$$
%\tau^i_m:=\sup\{ s>0: |u_i|_{L^q(0,s;H^ \delta_2)}<m\},\quad i=1,2,
%$$
%and put $\tau_m=\min(\tau_m^1,\tau^2_m)$. In case $q=\infty$ we will take
$$
\tau^i_m:=\sup\{ s>0: |u_i(s)|_{H^ \delta_2}<m\}\wedge T,\quad i=1,2.
$$
The aim is to show
that $u_1$ and $u_2$ are indistinguishable on the time interval $[0,\tau_m]$, {with $\tau_m=\inf(\tau^1_m, \tau^2_m)$.}

\medskip
%For $m\in\NN$ let $u_m$ denote the solution to \eqref{oben}, with
%
Fix  $m\in\NN$.
To get uniqueness on $[0,\tau_m]$ we first stop the original solution processes at time $\tau_m$ and extend the processes $u_1$ and $u_2$
by other processes to the whole interval $[0,T]$.  For this propose,
let $y_1$ be a solution to
\DEQSZ\label{eq1}
 y_1(t) &=& \CT(t-\tau_m ) u_1(\tau_m) %} && \\ &&{}
-
i \int_{\tau_m}^t \int_S \CT(t-s) y_1 (s) \, G(z)\,\tilde \eta(dz,ds)
\\
&&{}-i\int_{\tau_m}^t \int_S \CT(t-s) y_1 (s) H(z)\gamma(dz,ds)
,\quad t\ge \tau_m ,\nonumber  \EEQSZ
and let $y_2 $ be a solution to
\DEQSZ\label{eq2}
 y_2 (t) &= &\CT(t-\tau _m) u_2 (\tau _m) -
i\int_{\tau_m}^t \int_S \CT(t-s) y_2 (s) \,G(z)\,\tilde
\eta(dz,ds)
\\
&&{}-i\int_{\tau_m}^t \int_S \CT(t-s) y_2 (s) H(z)\gamma(dz,ds),\quad t\ge \tau_m
. \nonumber  \EEQSZ
Since $u_1$ and $u_2$ are \cadlag in $L^ 2(\RR^d)$, $u_1(\tau_m)$ and $u_2(\tau_m)$ are well defined and %. By estimate \eqref{finiteinlinfty} we know that %Since $u_1 $ and $u_2 $ are \cadlag on $[0,T]$ in $H^ \gamma_2(\RR^d)$, we know
%$u_1 (\tau_m)$ and $ u_2 (\tau _m)$
 belong $\PP$--a.s.\ to $L^2 (\RR^d)$.
Since, in addition, $(\CT(t))_{t\in\RR}$ is a strongly continuous group on $L^2 (\RR^d)$, %H^1_2(\RR^d)$,
the existence of unique solutions $y_1$ and $y_2 $ to \eqref{eq1} and \eqref{eq2} in $L^2(\RR^d)$
can be shown by standard methods.
Now, let us define two  processes $\bar u_1 $ and $\bar u_2$ which are
equal to $u_1 $ and $u_2 $ on the time interval $[0,\tau_m)$ and
follow the linear Schr\"odinger equation $y_1$ and $y_2$ afterwards.
In particular, let
$$
\bar u_1  (t) = \bcase u_1(t) & \mbox{ for } 0\le t< \tau_m,\\
y_1 (t) & \mbox{ for } \tau_m\le  t \le T,\ecase
$$
and
$$
\bar u _2  (t) = \bcase u_2 (t) & \mbox{ for } 0\le t< \tau_m,\\
y_2 (t) & \mbox{ for } \tau_m\le  t \le T.\ecase
$$
% such that
%
Note, that $\bar u_1$ and $\bar u_2$ solve
the truncated
equation corresponding  to  \eqref{eq_1}, that is
\begin{equation}
\begin{array}{l}
u(t)
 = \CT(t) u_0
 +  i\lambda \int_0^ t 1_{[0,\tau_m)}(s)\, \CT(t-s)
\lk( \lk|u(s)\rk|^{\alpha-1} \,
 u(s)\rk)\, ds\phantom{\Big|}
\\
\qquad {}- i \int_0^t\int_S  \CT(t-s) u(s)\, G(z)\, \tilde \eta(dz,ds)
\\
\qquad {}- i \int_0^t\int_S  \CT(t-s) u(s)\, H(z)  \, \gamma(dz, ds)\phantom{\Big|}.
\end{array} %-\frac i2 \int_0^ t \int_Z \CT(t-s) u(s) g(z) \eta(dz,ds).
\end{equation}
For $u_0\in L^2 (\RR^d)$ and $\xi\in \CN^2 (\Omega;L^2 ([0,T];L^2 (\RR^d)))$ progressively measurable with respect to the filtration $\BF$, %(\CF_t)_{t\in[0,T]}$,
let us define the integral operator
\DEQSZ\label{mZ}
\\
\nonumber (\mathfrak{Z}\xi)(t):= (\mT u_0)(t)+
\lk(\mathfrak{F}_{\tau_m}\xi\rk) (t)+\lk( \mathfrak{S}\xi \rk)(t)+\lk( \mathfrak{H}\xi\rk)(t),\quad t\in \INT,% \quad \in\mX.
  \EEQSZ
where
$\mT$ %:H^ {\sn}_{2}(\RR^d) \to \MX$
is defined by
$$(\mT u_0)(t) := \CT(t) u_0,
\quad  u_0\in L^ 2(\RR ^d), \quad t\in \INT,$$
the integral operator $\mathfrak{F}_{\tau_m}$ with respect to the nonlinear term  is defined by
\DEQS
( \mathfrak{F}_{\tau_m}  \xi )(t)&=&i\lambda \int_0^ {t} \CT({t}-s) \lk( |
\xi(s)|^{\alpha-1} \xi(s)\rk)\, 1_{[0,\tau_m)}(s)\, ds, \quad t\in[0,T]  ,
\EEQS
$\mG$ is defined by
\DEQS
%\label{eqn-stoch_ter_1}
(\FG \xi)(t)&=& %\left\{ \Rb{+}_0\ni t\mapsto
-i \int_0^ t \int_S \CT(t-s)
\xi(s) \, G( z)\; \tilde \eta(dz,ds), \quad t\in[0,T]  ,
%\right\}.
\EEQS
and $\mathfrak{H}$ is defined by
$$(\mathfrak {H} \xi) (t) := -i\int_0^t\int_S \CT(t-s)\xi (s) H(z) \gamma(dz,ds), \quad t\in[0,T]  .
$$
In the next step we will calculate the difference $\bar u_1 -\bar
u_2 $. % That is
%
%
%For $0\le s\le \tau$ we have
Fix $0\le t\le T$.  Similarly as before  we have
\DEQS \lqq{ \bar u_1 (t)- \bar u_2(t) =  (\mS  \bar u_1)(t)- (\mS  \bar u_2)(t)}&&
 \\
&&
+
(\mathfrak{H}  \bar u_1)(t)- (\mathfrak{H}  \bar u_2)(t)
+(\mF_{\tau_m}  \bar u_1)(t)- (\mF_{\tau_m} \bar  u_2)(t)
. \EEQS
Note, that
%Using again the operators $\mathfrak{F}$, $\mG$, and $\mathfrak{H}$, introduced at page \pageref{mZ}, we see
 $\mathfrak{G}$ and $\mathfrak{H}$ are linear.
In addition, $(\CT(t))_{t\in\RR}$ is  a unitary group on $L^2(\RR^d )$.
Therefore,
\DEQS
\EE \lk| \mathfrak{G}u_1(t)-\mathfrak{G}u_2(t)\rk|^2 _{L^2 } &\le &
\int_0^t \int_S \EE \lk|( \bar u_1(s)- \bar u_2(s)) G(z) \rk|^2_{L^2}  \,\nu(dz)\,  ds  .
\EEQS
Due to the integrability conditions on page \ref{integrabilitycondition}, i.e.\ the integrability condition i.),
% (see page \pageref{integrabilitycondition}),
 and the fact that $g$ is Lipschitz continuous, we know
that for any $v\in L^2(\RR^d)$ we have  $ \int_S |v\,G(z)|_{L^2}\nu(dz) \le C_g C_0(\nu)\, |v|_{L^2 }$ and we can proceed
\DEQS
\EE \lk| \mathfrak{G} \bar u_1(t)-\mathfrak{G} \bar u_2(t)\rk|^2 _{L^2 } &\le &
C_g C_0(\nu)   \int_0^t  \EE \lk| \bar u_1(s)- \bar u_2(s)  \rk|^2_{L^2}  \, ds.
\EEQS
Similarly, we get for  $\mathfrak{H}$ by the Minkowski inequality and the Lipschitz continuity of $h$,
\DEQS
\EE \lk| \mathfrak{H} \bar u_1(t)-\mathfrak{H} \bar u_2(t)\rk|^2 _{L^2 } &\le &
C_h  \, t  \int_0^t \int_Z \EE \lk|( \bar u_1(s)- \bar u_2(s)) H(z) \rk|^2_{L^2}  \,\nu(dz)\,  ds
\\
&\le&C\, C_h C_0(\nu)  T \int_0^t  \EE \lk| \bar u_1(s)- \bar u_2(s)  \rk|^2_{L^2}  \, ds.
\EEQS
The only term, which has to be carefully  analysed is the nonlinear term given by
\DEQS
\lqq{(\mF_{\tau_m}  \bar u_1)(t)- (\mF_{\tau_m} \bar  u_2)(t)}
&&
\\
 &=&i\lambda \int_0^ {t} \CT({t}-s) \lk( |\bar
u_1(s)|^{\alpha-1} \bar u_1(s)-|\bar u_2(s)|^{\alpha-1}
\bar u_2(s)\rk)\, 1_{[0,\tau_m)}(s)\, ds
.
\EEQS
%
%
%
%
% But before starting with the proof let us introduce the following constants.
Let $\gamma'$ and $\sigma$ be given such that $\tfrac 1\sigma+\tfrac 12=\tfrac 1 {\gamma'}$,
\DEQSZ\label{cccl}
-\delta + \frac d2<\frac d {\sigma(\alpha-1)},
\EEQSZ
and
\DEQSZ\label{gammaprime}
2\ge \gamma' \,\,\bcase \ge  1 \quad& \mbox{if}  \quad  d=1,\\ > 1 %>1 %{2d\over d+2}
& \mbox{if}\quad  d=2,
\\ \ge  {2d\over d+2}  & \mbox{if}\quad  d>2.
\ecase
\EEQSZ
Due to our assumption on $\alpha$ and $\delta$, such a couple $(\gamma',\sigma)$ exists.
Let $(\gamma,\rho)$ be an admissible pair, i.e.\ $2/\rho=d(1/2-1/\gamma)$.   %\footnote{for the definition of see \cite[p.\ 64]{Ponce}.}
Let   $\gamma$ be the  conjugate exponent to $\gamma'$, and
$\rho'$ the conjugate exponent to $\rho$.
 Then % and we will have
\DEQSZ\label{rhoprime2}
2\le \gamma\,\, \bcase \le \infty  \quad& \mbox{if}  \quad  d=1,\\ <\infty  %>1 %{2d\over d+2}
& \mbox{if}\quad  d=2,
\\  \le \frac {2d}{d-2}  & \mbox{if}\quad  d>2,
\ecase
\quad \mbox{and}&\quad
1\le \rho' \,\,\bcase \le \frac{4}3 \quad& \mbox{if}  \quad  d=1,\\ <{2} %>1 %{2d\over d+2}
& \mbox{if}\quad  d=2,
\\  \le 2  & \mbox{if}\quad  d>2.
\ecase
\EEQSZ
Before continuing, let us shortly introduce the Strichartz estimate.
Let us define the convolution operator
\DEQSZ\label{det-con}
\mT u (t) := \int_0^t \CT(t-r) u(r)\, dr, \quad t\ge 0.
\EEQSZ
Let $(p_0,q_0), (p_1,q_1)\in [2,\infty)\times[1,\infty)$ be two admissible pairs.
%satisfying the conditions \eqref{admissible}.
Then for all $T>0$  we have by the Strichatz estimate (see  \cite[p.\ 64]{Ponce})
$$
\lk| \mT u \rk| _{L^{q_0}  (0,T;L^{p_0})} \le
C \lk|  u \rk| _{L^{q_1'}  (0,T;L^{p_1'})}.
$$
%
%To be more precise, $\gamma=\tfrac{2(1+\alpha)}{3-\alpha}$, $\rho=\tfrac{2(1+\alpha)}{d(\alpha-1)}$ and $\rho'=\tfrac{2(1+\alpha)}{2+2\alpha+d-\alpha d}$.
Applying the  Strichartz estimate  we have, {for any $r\ge 1$,} %for $r=q/(q-(\alpha-1)\rho')$
\DEQS
\lqq{ \sup_{0\le t\le T} \lk| % \mathfrak{F}_1
\int_0^ {t} \CT({t}-s)\lk( |\bar
u_2(s)|^{\alpha-1} \bar u_2(s)-|\bar u_1(s)|^{\alpha-1}\bar u_1(s)\rk)\, 1_{[0,\tau_m)}(s)\, ds
\rk|^{\rho'r}
_{L^2}} &&
\\
&\le&%%
  \lk|\lk(|\bar
u_2(s)|^{\alpha-1} \bar u_2(s)-|\bar u_1(s)|^{\alpha-1}\bar u_1(s)\rk)1_{[0,\tau_m)}(s)\rk|_{L^ {\rho'}([0,T];L^{ \gamma'}) } ^{\rho'r}.
\EEQS
The H\"older inequality gives for $\sigma$ with $\frac 1 \sigma+\frac 12 =\frac 1{\gamma'}$, i.e.\
\DEQSZ\label{gammaprime1}
\sigma \,\,\bcase \ge  2 \quad& \mbox{if}  \quad  d=1,\\ >2 %>1 %{2d\over d+2}
& \mbox{if}\quad  d=2,
\\ \ge  d & \mbox{if}\quad  d>2,
\ecase
\EEQSZ
\DEQS
\ldots &\le&%%
\lk( \int_0^T   \lk|(|\bar
u_2(s)|^{\alpha-1} +|\bar u_1(s)|^{\alpha-1})1_{[0,\tau_m)}\rk| _{L^\sigma}^{\rho'}  \lk| (\bar u_1(s)-\bar u_2(s)) 1_{[0,\tau_m)}(s) \rk|_{L^2 }^{\rho'}
\, ds\rk) ^{r} %\frac 1 {\rho'}
\\
&\le&%%
\lk( \int_0^T   \lk|(|\bar
u_2(s)| +|\bar u_1(s)|)1_{[0,\tau_m)}\rk| _{L^{(\alpha-1) \sigma}}^{(\alpha-1)\rho'}  \lk| (\bar u_1(s)-\bar u_2(s)) 1_{[0,\tau_m)}(s) \rk|_{L^2 }^{\rho'}
\, ds\rk) ^ r %\frac 1 {\rho'}
.
\EEQS
By inequality \eqref{cccl} and Sobolev embeddings Theorems  we know that
$H^ \delta_2(\RR^ d )\hookrightarrow L^ {\sigma(\alpha-1)}(\RR^ d )$ continuously.
This implies
\DEQS
\ldots &\le&%%
\lk( \int_0^T   \lk|(|\bar
u_2(s)| +|\bar u_1(s)|)1_{[0,\tau_m)}\rk| ^{(\alpha-1)\rho'} _{H^ \delta_2}  \lk| (\bar u_1(s)-\bar u_2(s)) 1_{[0,\tau_m)}(s) \rk|_{L^2  }^{\rho'}
\, ds\rk) ^r . %\frac 1 {\rho'}.
\EEQS
By  the H\"older inequality %we know that there exists some $\rho^\sharp>0$ such that  for $q>4(\alpha-1)$ for $d=1$, and $q=\infty$ for $d\ge 3$
{\DEQS
\ldots &\le&%%
 C(T)  \sup_{s\in [0,T]} \lk|(|\bar
u_2(s)| +|\bar u_1(s)|)1_{[0,\tau_m)}(s)\rk| _{H^ \delta_2}^{(\alpha-1)\rho' r} \,
\\
&&\hspace{2cm} {}\times  \int_0^T   \lk| (\bar u_1(s)-\bar u_2(s)) 1_{[0,\tau_m)} (s)\rk|_{L^2  }^{r\rho'}
\, ds . %\frac 1 {\rho'}.
\EEQS}
The definition of the stopping time gives
\DEQS
\ldots &\le&%%
 m ^{(\alpha-1)\rho' r}   \int_0^T   \lk| (\bar u_1(s)-\bar u_2(s)) 1_{[0,\tau_m)} (s)\rk|_{L^2  }^{r\rho'}
\, ds . %\frac 1 {\rho'}.
\EEQS
Now, taking $r=2/\rho'$, an application of the  Grownwall Lemma gives $\EE | \bar u_1(t)- \bar u_2(t)|^2_{L^2 }=0$. Since
$ \bar u_1$ and $ \bar u_2$ are \cadlag on  $L^2 (\RR^d)$,
%$H_2^\sn(\RR^d)\hookrightarrow L^2 (\RR^d)$ continuously,
 both processes
$ \bar u_1$ and $ \bar u_2$ are  indistinguishable in $L^2 (\RR^d)$ on the time interval $[0,\tau_m]$.

%%%%%%%%%%%%%%%%%%%%%%%%%%%%%%%%%%%%%%%%%%%%%%%%%%%%%%%%%%%%%%%%%%%%%%%%%%%%%%%%%%%%%%%%%%%%%%%%%%%%%%%%%%%%%%%%%%%%%%%%%%%%%%%%%%%%%%%%%%

\medskip

\paragraph{\bf Step II:}
We show that $\PP\lk( \tau_m<T\rk) \to 0$ as $m\to\infty$.
Observe, that
it holds for $\delta$
$$
\{ \tau_m\le T \} \subset \{ |u_1|_{L^{\infty}([0,T];H^{\delta}_2)}\ge m \; \mbox{or} \; |u_2|_{L^{\infty}([0,T];H^\delta_2)}\ge m\}.
$$
%Due to  Sobolev embedding Theorems gives $H^1_2(\RR^d)\hookrightarrow L^p(\RR^d)$ and therefore,
% $$
% |u_1 |_{L^q (0,T;L^p)}\le |u_1 |_{L^q (0,T;H^1_2)}\;\mbox{and}\; |u_2 |_{L^q (0,T;L^p)}\le |u_2 |_{L^q (0,T;H^1_2)}
%,
%$$
Therefore,
$$
\PP\lk( \tau_m <  T \rk) \le \PP\lk( |u_1|_{L^{\infty} ([0,T];H^\delta_2)}\ge m \rk) + \PP\lk( |u_2|_{L^{\infty} ([0,T];H^\delta_2)}\ge m \rk).
$$
Since $u_1$ and $u_2$ are \cadlag in $H^\delta_2(\RR^d)$, and $\PP$-a.s.\ $ \sup_{0\le s\le T}|u_1(s)|_{H^\delta_2}<\infty$ and
$ \sup_{0\le s\le T}|u_2(s)|_{H^\delta_2}<\infty$, it follows
$$
\PP\lk(  |u_i|_{L^ \infty([0,T]; H^\delta_2)}\ge m \rk)\longrightarrow 0,
$$
as $m\to\infty$, for $i=1,2$. This implies  $\PP\lk( \tau_m \le  T \rk)\to 0$ as $m\to\infty$.
Hence,  both processes $u_1$ and $u_2$ are undistinguishable on $[0,T]$.
\end{proof}

%%%%%%%%%%%%%%%%%%%%%%%%%%%%%%%%%%%%%%%%%%%%%%%%%%%%%%%%%%%%%%%%%%%%%%%%%%%%%%%%%%%%%%%%%%%%%%%%%%%%%%%%%%%%%%%%%%%%%%%%%%%%%%%%%%%%%

%%%%%%%%%%%%%%%%%%%%%%%%%%%%%%%%%%%%%%%%%%%%%%%%%%%%%%%%%%%%%%%%%%%%%%%%%%%%%%%%%%%%%%%%%%%%%%%%%%%%%%%%%%%%%%%%%%%%%%%%%%%%%%%%%%%%%%%%%%%

%%%%%%%%%%%%%%%%%%%%%%%%%%%%%%%%%%%%%%%%%%%%%%%%%%%%%%%%%%%%%%%%%%%%%%%%%%%%%%%%%%%%%%%%%%%%%%%%%%%%%%%%%%%%%%%%%%%%%%%%%%%%%%%%%%%%%%%%%%%

%

Due to the pathwise uniqueness one can show that a unique strong solution exists.

\begin{theorem}\label{main_example}
If the conditions of Theorem \ref{main} and Theorem \ref{main_ex_u} are satisfied,
then there exists a unique strong solution to equation  \eqref{eq_1} in $\DD([0,T];L^2(\RR^d))$.
\end{theorem}

\begin{proof}
Let $\delta<1$ be the constant given in Theorem \ref{main_ex_u}.
In order to apply the Theorem \ref{thm513} below, {we put} $Y=H^{-30 d}_2(\RR^d )$, {$Y_0=\mathscr D(\Bbb R^d)$},
 $X=H^ 1_2 (\RR^d)$, and $$\CX:= \DD([0,T];L^2 (\RR^d)).$$% H^ \delta(\RR^d)).$$
 %{\color{Blue}$Y=H^ L^1_{loc}(\RR^d )$, $X=H^1_2(\RR^d)$ and $\CX=\DD([0,T]; H^\delta_2)$}.
 %H^{\delta }_2(\RR^d ))
 %L^q\lk([0,T];H^1_2(\RR^d)\rk)\blufy{\cap \DD([0,T];L^2(\RR^d))}
 %.$$
We have chosen $Y=H^{-30 d}_2(\RR^d )$, since $Y$ can be arbitrarily large, only we have to be sure that following points has to be satisfied
and these points are satisfied by this choice.
\noindent
The functions
\begin{itemize}

\item $a:{[0,T]}\times X\to Y$,
\item $b:{[0,T]}\times{[0,T]}\times X\to Y$,
\item $c:{[0,T]}\times{[0,T]}\times S\times X\to Y$,
\item {$\theta^{\alpha_i}_i:\Bbb D([0,T],Y)\to[0,\infty]$, $\alpha_i\in A_i$, $i\in\{0,1\}$, }
\item $\{ \Xi_t:t\in [0,T] %\RR_0
\}$, where $\Xi_t:\Bbb D([0,T],Y)\to[0,\infty]$
%\item {{$\Xi_t:D([0,T];Y)\to [0,\infty]$,}} %D([0,T],H^\delta(\RR^d ))
\end{itemize}
are defined by
\begin{itemize}
\item $a(t,x):= \CT(t) x,$ for $t\in[0,T]$ and $x\in X$;
\item $b(t,s,x) := \CT(t-s)\, 1_{[0,t)}(s)\,F_\alpha(x)+ \int_S  \CT(t-s) 1_{[0,t)}(s)\,{[ x  H(z)]}\nu(dz)$, where $F_\alpha(x)= |x|^{\alpha-1}x$,
for $s,t\in[0,T]$, $x\in X \subset L^{\alpha+1}(\RR^d)$, so that $F_\alpha(x) %|x|^{\alpha-1}x
\in L^{\alpha+1/\alpha}(\RR^d)$;
\item $c(t,s,z,x) := \CT(t-s) \, 1_{[0,t)}(s)\,{(x \,G(z))},
$ for $s,t\in[0,T]$, $x\in X$ and $z\in S$;

\item $A_0=\{1\}$, $A_1=\emptyset$,
$$\theta^1_0:\Bbb D([0,T],Y) %\Bbb \CX %D([0,T],Y)
\ni u \mapsto \bcase \infty &\mbox{ if } u\not \in \DD(0,T;{H^\delta _2}(\RR^d))\\  1 &\mbox{ if } u\in \DD(0,T;{H^\delta _2}(\RR^d)), \ecase
$$
\item
{
For all $\phi\in Y^\ast$ and $t\in[0,T]$, $\Xi_t^\phi(u)=\infty$ for $u\in\Bbb D([0,T];Y)\setminus\Bbb D([0,T];L^2(\Bbb R^d))$ and, for $u\in D([0,T];L^2(\Bbb R^d))$ %for all $\phi\in Y^ \ast$
\DEQS
\lqq{
\Xi^\phi_t(u)=\int_0^ t\lk\{\la \CT(t-s)F_{\alpha}(u(s)),\phi\ra  +\int_{\{y\in S: |\la  \CT(t-s)  [u(s)\, G(y)],\phi\ra|<1\}}\la  \CT(t-s)  [u(s)\, G(y)],\phi\ra ^ p\rk.
}&&
\\
&&{}\lk. +\int_{\{y\in S: |\la  \CT(t-s)  [u(s)\, G(y)],\phi\ra|\ge 1 \}}
\la  \CT(t-s) [ u(s)\, G(y)], \phi \ra  +\la \CT(t-s) [ u(s)\, H(y)],\phi\ra \nu(dy)\right\} \, ds
\EEQS
for $u\in\Bbb D(0,T;L^2(\Bbb R^d))$, where $F_\alpha(x)$ is as above.

%{\color{Red} This integrability condition must be the same as the integrability condition in Definition 3.1, so that we can apply Section 4 to the equation in Section 3.}

}
\end{itemize}

\begin{remark}
Roughly speaking, the setting in Theorem \ref{main_example} has to fit to the  setting in Section \ref{sec_abstr} and has to be chosen as follows.
The space $X$ has to be
a space such that for any $t\ge 0$, $u(t)$ is {an} $X$--valued random variable. Note, that this space need not coincide with the space where
the process is \cadlag.
The space $Y$ has to be chosen such that the mappings $a$, $b$ and $c$ are well defined, $\DD([0,T];Y)$ corresponds to the Kurtz space (denoted in his paper by {$Z_1$}),
and $\CX$ is the path space.
Additional regularity properties of the solution can be  incorporated in the family of functions
{$\{\theta_0^{\alpha_0}:\alpha_0\in A_0\}$ and $\{\theta_1^{\alpha_1}:\alpha_1\in A_1\}$, where $A_0$ and $A_1$} are  index sets. In our  case, we choose $A_1={\{1\}}$ and $A_2=\emptyset$.
To make sure that the integrals in \eqref{finiteintegrals}
are well defined, the family of functions {$\{\Xi_t:t\in(0,T],\phi\in Y^ \ast\}$} has to be defined in a proper way.
\end{remark}
\medskip

\begin{remark}
The setting need not to be unique. Thus, instead of the setting
above, we could chose $X=H^1_2(\RR^d)$, $\CX=\CX_0= D([0,T];H_2^ \delta(\RR^d))$, $A_1=\emptyset$, and $A_2=\emptyset$.
Indeed, due to the fact that given pathwise uniqueness in $L^ 2(\RR^d)$ with the condition $\PP\lk( u\in \DD([0,T];H_2^ \delta(\RR^d)\rk)=1$
one has pathwise uniqueness in $H^\delta_2(\RR^d)$ and we could take directly $\CX= D([0,T];H^ \delta_2(\RR^d))$.

\end{remark}
\medskip

In the next step, we have to verify that the mappings $a,b$ and $c$ are measurable.
In fact, first note  that $\CT$ is a strongly continuous unitary  group on $H^{-30d}_2(\RR^d )$, therefore, $a$ is measurable.
To show that $b$ is measurable, we first investigate the measurability of $\CT(t-s)\, 1_{[0,t)}(s)\,F_\alpha(x)$.
%If $d=1$, then $H^ 1 _2(\RR^d) \hookrightarrow L^ \infty(\RR^d)$ and $F_\alpha:H^ 1 _2(\RR^d)\to L^ \infty(\RR^d)\cap L^ 2 (\RR^d)$ is continuous.
Here, we show that $F_\alpha:H^1_2 (\RR^d)\to Y$ is continuous, from which follows that $F_\alpha:H^ 1 _2 (\RR^d)\to Y$ is measurable.
Indeed, using H\"older inequality, it is easily seen that for any $x,y \in L^{\alpha+1}(\RR^d)$,
$$
|F_{\alpha}(x)-F_{\alpha}(y)|_{L^{\alpha+1/\alpha}(\RR^d)}\le C \lk( |x|_{L^{\alpha+1}(\RR^d)}^{\alpha-1}+|y|_{L^{\alpha+1}(\RR^d)}^{\alpha-1}\rk)
|x-y|_{L^{\alpha+1}(\RR^d)},
$$
and the continuity result follows from the embeddings $H^1_2(\RR^d)\hookrightarrow L^{\alpha+1}(\RR^d)$ and $L^{\alpha+1/\alpha} (\RR^d)\hookrightarrow  Y$.
Since $1_{[0,t)}(s)$ is measurable, and $(\CT(t))_{t\in\RR}$ is pointwise strongly continuous, we are done.
It remains to show that the second term of $b$ is measurable, but this follows since, by the Lipschitz continuity of $h$,
and the fact that $h(0)=0$, the  Nemitsky operator $H$ is continuous  from $L^2 (\RR^ d)$ to $Y$.
Similarly the last mapping $c$ can be handled.
It remains to verify that
$$\theta_0: Y %D([0,T],Y)
\ni u \mapsto \bcase \infty &\mbox{ if } u\not \in \DD([0,T];{H^\delta _2}(\RR^d)),\\  1 &\mbox{ if } u\in \DD([0,T];{H^\delta _2}(\RR^d)),\ecase
$$
is a measurable mapping. But this is given, since $\DD([0,T];{H^\delta _2}(\RR^d))$ is a Borel subset of $\DD([0,T];Y)$.
Now, the existence of the strong solutions follows by an application of Theorem \ref{thm513}.

\end{proof}

\section{The abstract {uniqueness} result}\label{sec_abstr}

Let $X$ and $Y$ be separable Fr\'echet spaces, \refy{$X\hookrightarrow Y$ densely}, $Y_0\subseteq Y^*$ separates points in $Y$, $\mathcal X$ a Borel subset in $\Bbb D([0,T],Y)$, $A_1$ and $A_2$ are two index sets and
\begin{itemize}
\item $a:{[0,T]}\times X\to Y$, $b:{[0,T]}\times{[0,T]}\times X\to Y$,
\item $c:{[0,T]}\times{[0,T]}\times {S}\times X\to Y$,
\item {$\theta^{\alpha_i}_i:\Bbb D([0,T],Y)
\to[0,\infty]$, $\alpha_i\in A_i$, $i\in\{0,1\}$, }
\item {{$\{ \Xi^\phi_t:t\in[0,T],\phi\in Y^ \ast \}$, where $\Xi_t:\Bbb D([0,T],Y)\to[0,\infty]$.}}
\end{itemize}
measurable mappings, $\nu$ a $\sigma$-finite measure on $(S,\CS)$, and, finally, $S_n\in\CS$ such that $S_n\uparrow S$ and $\nu(S_n)<\infty$.

\del{In addition, let us define $$\CX_0:= \lk\{ u\in\CX: {\Xi_t(u)<\infty, \,\,t\in[0,T]} %\mbox{ and } \EE\, \theta^ {\alpha_2}_0(u)<\infty, \alpha_2\in A_2
 \rk\}.
$$
\begin{remark}
The space $\mathcal X_0$ comprises some regularity properties of the solution. That is, in our example we will have { $\mathcal X_0=\Bbb D([0,T],H^ \delta_2 (\RR^d))$, $\delta$ given in \eqref{givendelta}, $Y=H^ {-30 d}_{2}(\RR^d)$ and $X=L^ 2 (\RR^d )$.}
\end{remark}
}

Let $\eta$ be a {time homogeneous} Poisson random measure {with the intensity measure $\nu$} on
{the space $(S,\CS)$} defined over a  probability space $\MA=(\Omega,\CF,\BF,\PP)$, where  $\BF$ denotes a  filtration $(\CF_t)_{t\in\INT}$.
Given is an abstract evolution equation of the following form:
\begin{eqnarray}
\langle u(t),\varphi\rangle&=&\langle a(t,u(0)),\varphi\rangle+\int_0^{t}\langle b(t,s,u(s)),\varphi\rangle\,ds\label{eq}
\\
&+&\int_0^{t}\int_S\langle c(t,s,x,u(s)),\varphi\rangle\, \tilde\eta(dx,ds),\nonumber
\end{eqnarray}
for every $t\in{[0,T]}$ and $\varphi\in{Y_0}$.
%
%
%
%Here $\eta$ is a Poisson random measure on $Z$ over a probability space $\MA=(\Omega,\CF,(\CF_t)_{t\ge 0},\PP)$.
We define next the terminus of solution in the way we will use it in the following pages of the article.
\begin{definition}\label{def_solution}
We say that a $6$-tuple $(\Omega,\CF,\BF, \PP,  u,\eta )$, $\BF=(\CF_t)_{t\in \INT}$,
consisting of a filtered probability space $\mathfrak{A}=(\Omega,\CF,\BF, \PP)$, a time homogeneous Poisson random measure $\eta$ on $(S,\CS)$ over $\mathfrak{A}$ {with the intensity measure $\nu$} and a  process $u$ on $[0, T]$, being {$\BF$-adapted and} c\`adl\`ag in {$Y$},
is a solution of \eqref{eq} provided that
$$
\Bbb P( u(t)\in X) =1, \quad  \forall t\in{[0,T]},\quad\qquad\PP\lk( u\in{\mathcal X} \rk) =1,
$$
for all $t\in\INT$
\DEQSZ\label{finiteintegrals}
\lqq{ \int_0^{t}|\langle b(t,s,u(s)),\varphi\rangle|\,ds+
\int_0^{t}\int_{\{x\in S: |\langle c(t,s,x,u(s)),\varphi\rangle|< 1\}}|\langle c(t,s,x,u(s)),\varphi\rangle|^p\wedge 1\,\nu(dx)\,ds} &&
\\ \nonumber
&& +
\int_0^{t}\int_{\{ x\in S: |\langle c(t,s,x,u(s)),\varphi\rangle|\ge 1\}}  |\langle c(t,s,x,u(s)),\varphi\rangle|\,\nu(dx)\,ds <\infty\quad\textrm{$\PP$--a.s.}
\EEQSZ
hold for every $t\in[0,T],$ {$\varphi\in Y_0$, %$\alpha_1\in A_1$, $\alpha_2\in A_2$}
}and $u$ solves   equation \eqref{eq}.
\end{definition}
\begin{remark}
Additional regularity properties, which are not part of the definition of the solution, but %first are satisfied by the solution and
which are essential for the  pathwise uniqueness, are incorporated by the additional mappings {$\theta^ {\alpha_0}_0$ and $\theta^ {\alpha_1}_1$, $\alpha_0\in A_0$, $\alpha_1\in A_1$}.
\end{remark}
\begin{hypo}\label{hyp_solution}
%{\color{Red}(I suggest to rename this definition a Hypothesis (HYP))}
A solution satisfies the additional regularity properties given by  % $\CA_1$ and $\CA_2$,
%if there are two family of functions
{$\{\theta_0^{\alpha_0}:\alpha_0\in A_0\}$ and $\{\theta_1^{\alpha_1}:\alpha_1\in A_1\}$, where $A_0$ and $A_1$} are  index sets,
such that
\begin{equation}\label{hyp_solution-eqn}\tag{HYP}
\PP\lk(\theta_0^{{\alpha_0}}(u)<\infty\rk)=1, \quad\textrm{ and }\quad\Bbb E\,\theta_1^{{\alpha_1}}(u)<\infty.
%\footnote{this properties are not part of the definition of the solutiuon}
\end{equation}
\end{hypo}

\begin{definition} If $\eta\in{M_{\bar{\Bbb N}}(\{S_n\times\Bbb R_+\})}$ then we define
\begin{equation}\label{etat}
\eta_t(V)=\eta(V\cap({S}\times(0,t])),\,\eta^t(V)=\eta(V\cap({S}\times(t,T] )),\,V\in{\mathcal S\otimes\mathcal B(\Int )}.
\end{equation}
\end{definition}

{
\begin{lemma}\label{adapttimhom} If $\eta$ is a time homogeneous Poisson random measure over a filtered probability space $\mathfrak{A}=(\Omega,\CF,\BF, \PP)$, then, for every $t\in \INT $, $\eta_t$  is an $\mathcal F_t$-measurable {$M_{\bar{\Bbb N}}(\{S_n\times\Bbb R_+\})$}-valued random variable and $\eta^t$ is  independent of $\mathcal F_t$.
\end{lemma}

\begin{proof} The first assertion of the Lemma follows directly from the definition. In particular,
since $\eta$ is adapted to the filtration $\BF$, the first assertion follows.
The second assertion follows from the independently scattered property. In particular, for any $n\in\NN$, $U,V\in\CS$, $s\in(0,t]$ and $r_1,r_2\in(t,\infty)$ %and $s\in (t,T]$
the sets $U\cap S_n\times [0,s]$ and $V\cap S_n\times (r_1,r_2)$ are disjoint, therefore the $\sigma$--algebra generated by the random variables $\{\eta_t(U\cap S_n):U\in\CS\}$
and $\{ \eta^t (U\cap S_n): U\in \CS\}$ are independent. Hence, the filtration generated by $\eta^t $ and the $\sigma$--algebra generated by $\eta_t$ are independent.
\end{proof}
}

\medskip

\begin{lemma}\label{transequal} Let $\mathfrak{A}=(\Omega,\mathcal F,\BF %(\mathcal F_t)_{t\ge 0}
,\Bbb P)$ be a filtered probability space  with filtration $\BF=(\mathcal F_t)_{t\in \INT }$, $u$ a $\Bbb D({[0,T]};Y)$-valued random variable over $\mathfrak{A}$,
and $\eta$ an {$M_{\bar{\Bbb N}}(\{S_n\times\Int \})$}-valued random variable over $\mathfrak{A}$. In addition,  we assume that  for any $t\ge 0$,   $u(t)$  and $\eta_t$ are $\mathcal F_t$-measurable and  $\eta^t$ is independent of $\mathcal F_t$.
If there exists a solution $\overline{ \mathfrak{A}}=(\bar\Omega,\bar{\mathcal F},(\bar{\mathcal F}_t)_{t\ge 0},\bar{\Bbb P},\bar u,\bar\eta)$ of Equation \eqref{eq} satisfying Equation \eqref{eq} and the assumption of Definition \ref{def_solution},  %Hypothesis \ref{hyp_solution}}, in particular \eqref{hyp_solution-eqn},
such that the law of $(u,\eta)$ coincides with the law of $(\bar u,\bar\eta)$ on $\DD({[0,T]};Y)\times {M_{\bar{\Bbb N}}(\{S_n\times\INT \})}$, then $(\Omega,\mathcal F,(\mathcal F^{\Bbb P}_t)_{t\ge 0},\Bbb P,u,\eta)$ is a solution to Equation \eqref{eq} {satisfying Hypothesis \ref{hyp_solution}}.
\end{lemma}

Before giving the proof of  {Lemma \ref{transequal}}, let us introduce some notations. For jointly Borel measurable mappings $\breve a:\Bbb R^+_0\times X\to\Bbb R$, $\breve b:\Bbb R^+_0\times\Bbb R^+_0\times X\to\Bbb R$ and $\breve c:\Bbb R^+_0\times\Bbb R^+_0\times Z\times X\to\Bbb R$, a time homogeneous Poisson random measure $\eta$ {with the intensity measure $\nu$} and an {adapted c\`adl\`ag} process $v$ {in $Y$ with $v(t)\in X$ a.s.\ for every $t\in\RR^+_0$}, both defined on a probability space $\mathfrak{A}=(\Omega,\mathcal F,\BF,\Bbb P)$ with filtration $\BF=(\mathcal F_t)_{t\in\INT}$, {satisfying}
\DEQS
\lqq{ \int_0^\infty|\breve b(t,s,v(s))|_Y\,ds+\int_0^\infty\int_{\{x\in S:|\breve c(t,s,x,v(s))|_Y\le 1\}}|\breve c(t,s,x,v(s))|_Y^p\,\nu(dx)\,ds}
&&
\\
&&{}+\int_0^\infty\int_{\{x\in S:|\breve c(t,s,x,v(s))|_Y>1\}} |\breve c(t,s,x,v(s))|_Y\,\nu(dx)\,ds
<\infty\quad\textrm{a.s.}
\EEQS
for every $t\in\Int $, define a nonlinear map  $\CK_\mathfrak{A}$
\begin{eqnarray}\label{definition-ck}
\CK_{\mathfrak{A}}(v,\eta) (t) &=& \breve a(t,v_0)+\int_0^\infty\breve b(t,s,v(s))\,ds
\\
&+&\int_0^\infty\int_S\breve c(t,s,x,v(s))\,\tilde\eta(dx,ds).\nonumber
\end{eqnarray}
Observe that $\CK_{\MA}(v,\eta)$ actually depends via the compensator of $\tilde \eta$,   also on the probability measure $\PP$.

\medskip

\del{
\begin{remark}
Let $\BF^\eta=(\CF^\eta_t)_{t\in[0,T]}$ be the filtration induced by the Poisson random measure $\eta$ on the probability space $(\Omega,\CF,\PP)$.
By the same considerations as done in Step III and Step IV in \cite{reacdiff} one can see that
$u$ is progressively measurable with respect to $\BF^\eta$.
\end{remark}
}
\medskip

\begin{proof}[Proof of  {Lemma \ref{transequal}}:]

{It is rather standard to prove that $\eta$ is a time homogenous Poisson random measure with intensity $\nu$ for the augmented filtration $(\CF_t^{\Bbb P})_{t\ge 0}$, cf.\ Lemma \ref{adapttimhom}, and that all the measure and integrability assumptions in Definition \ref{def_solution} are satisfied for $\Bbb P$, $u$ and $\eta$. So it just remains to prove that the actual equation \eqref{eq} holds, i.e. that}
% \label{defsolution}
for any $t\in\RR^+_0$ {and $\varphi\in Y^*$} we have
$${\PP} \lk(\CK_\mathfrak{A}(u,\eta)(t)-\langle u(t),\varphi\rangle=0\rk)=1,
      $$
      where  {$\CK_\mathfrak{A}$ is defined with $\breve a=\langle a,\varphi\rangle$, $\breve b=\langle b,\varphi\rangle$ and $\breve c=\langle c,\varphi\rangle$};
%\end{romanpoint}

\medskip
\noindent
Fix $t\in\RR_0^+$ and $\varphi\in Y^\ast$.
Let us remind, that the mappings
\begin{itemize}
\item $a:{[0,T]}\times X\to Y$, $b:{[0,T]}\times{[0,T]}\times X\to Y$,
\item $c:{[0,T]}\times{[0,T]}\times {S}\times X\to Y$,
\item {$\theta^{\alpha_i}_i:\Bbb D([0,T],Y)
\to[0,\infty]$, $\alpha_i\in A_i$, $i\in\{0,1\}$, }
%\item $\theta_0^1:\Bbb D([0,T];Y)\to[0,\infty]$,
\item $\Xi_t^\varphi:[0,T]\times\Bbb D([0,T];Y)\to [0,\infty]$, %D([0,T],H^\delta(\RR^d ))
%\item $\theta_i:{[0,T]}\times\Bbb D({[0,T]},Y)\to[0,\infty]$, $i=0,1$
\end{itemize}
are measurable.
Hence, the mapping
$$
X\ni v_0\mapsto {\mathfrak a}_t:= %\breve a(t,\bar v_0)=
\la a(t,v_0
),\varphi\ra \in \RR
$$
is Borel measurable.
 Since $u$ and $\bar u$ belongs a.s.\ to $\DD([0,T];Y)$ and $u(0)$, $\bar u(0)$ belongs a.s.\ to $X$,
$u(0)$ and $\bar u(0)$ have the same law on $X$, it follows by Lemma 1.22 \cite{kallenberg} that  the triplets
$(   {\mathfrak a}_t(u(0)), u,\eta)$ and $( {\mathfrak a}_t(\bar u(0)) ,\bar u,\bar \eta)$ have the same law on $X\times \DD(\INT;Y)\times\CMM$.

Since $b:{[0,T]}\times{[0,T]}\times X\to Y$, is measurable, for any $s\in[0,T]$ $u(s)$ and $\bar u(s)$ are $X$--valued random variables,  $\Law(\bar u(s))=\Law(u(s))$, it follows that
for any $s,t\in[0,T]$ and $\omega\in $
the processes $\{\overline{ \mathfrak{b}}_t(s): s\ge 0\}$ and $\{ \mathfrak{b}_t(s): s\ge 0\}$, defined by
$$
\overline{ {\mathfrak b}}_t(s,\omega) := \langle b(t,s,\bar u(s,\omega),\varphi\rangle ,\quad s\in [0,T],
$$
and
$$
{ {\mathfrak b}}_t(s,\omega) := \la b(t,s, u(s,\omega),\varphi\ra  ,\quad s\in [0,T],
$$
have the same law for each $t\in[0,T]$.
In addition, we know by the definition of the solution that
$$\bar \PP\lk( \int_0^ t |\overline{  {\mathfrak  b}}_t(s)|\, ds <\infty\rk)=1,
\quad\mbox{and} \quad
\PP\lk( \int_0^ t | {\mathfrak b}_t(s)|\, ds <\infty\rk)=1.
$$
By Theorem 8.3 of \cite{martin1}, it follows %that given $\Law( \breve b(t,s,u_1(s)))=\Law( \breve b(t,s,u_1(s)))$,
that for any $t\in[0,T]$ %then
$$\Law\lk(\bar u,\int_0^ t \bar {\mathfrak  b}_t(s)ds\rk)
\quad\mbox{and} \quad
\Law\lk(u,\int_0^ t  {\mathfrak  b}_t(s)ds\rk)$$
 are equal on $\DD([0,T];L^ 2(\RR^d))\times Y$.
Finally,
since $c:{[0,T]}\times{[0,T]}\times {S}\times X\to Y$
is measurable, for any $s\in[0,T]$, the random variable $\bar u(s)$ is  $\bar \CF_s$--measurable, and the random variable  $u(s)$ is $\bar \CF_s$--measurable.
It follows that the processes $\{\overline{ \mathfrak{c}}_t(s): s\ge 0\}$ and $\{ \mathfrak{c}_t(s): s\ge 0\}$, defined by
$$
\overline{ {\mathfrak c}}_t(s,\omega) := \langle c(t,s,\bar u(s,\omega),\varphi\rangle ,\quad s\in [0,T],
$$
and
$$
{ {\mathfrak c}}_t(s,\omega) := \la c(t,s, u(s,\omega),\varphi\ra  ,\quad s\in [0,T],
$$
are  adapted to the filtrations $(\bar \CF_{s})_{s\in[0,T]}$ and  $(\CF_{s})_{s\in[0,T]}$, respectively. In addition, since due to the fact that $\bar u$ is a solution, and the law of $(\bar u,\bar \eta)$ coincides with the law of $(u,\eta)$,
we know that
%\DEQSZ\label{finiteintegrals}
$$\bar \PP\lk( \int_0^ t \int_{\{x\in S: |\langle \bar c(t,s,x,u(s)),\varphi\rangle|< 1\}} |\bar {\mathfrak  c}_t(x,s)|^ p\, \nu(dx)\, ds <\infty\rk)=1,
$$
and, therefore,  %\quad\mbox{and} \quad
$$ \PP\lk( \int_0^ t \int_{\{ x\in S: |\langle c(t,s,x,u(s)),\varphi\rangle|<  1\}}   | {\mathfrak c}_t(x,s)|\, \nu(dx)\, ds <\infty\rk)=1.
$$
Hence,
$\bar \PP$--a.s.\ the process $[0,T]\ni s\mapsto { \bar{{\mathfrak c}}}_t(s)\in\RR$ and $\PP$--a.s.\ the process
$[0,T]\ni s\mapsto { {\mathfrak c}}_t(s)\in\RR$)
are belonging (but of the large jumps) to
$  L^p ( \INT ;\RR)$. It follows by Proposition \ref{cont-haar-shift}-(ii) that they are progressively measurable.
Hence, %It now follows from
  Theorem \ref{equality}  is applicable and we know that
 $(\bar \CI(t),\bar u,\bar \eta)$  and $(\CI(t),u,\eta)$ have the same law on $\RR\times \DD([0,T];Y)\times \CMM)$,
where
$$
\bar \CI (t) :=  \int_{0}^ T\int_S \breve c (t,s,\bar u(s),z)\, \tilde \eta(dz,ds), \quad t\in[0,T],
$$
and
$$
 \CI (t) :=  \int_{0}^ T \int_S\breve c (t,s,u(s),z)\, \tilde \eta(dz,ds), \quad t\in[0,T].
$$

To deal with the large jumps, we use the fact that
$$
\bar \PP\lk( \int_0^t \int_{\{x\in S:|\overline{{\mathfrak c}}_t(x,s)|\ge 1\}} |\bar{{\mathfrak c}}_t(x,s)|\,\nu(dx)\,ds<\infty\rk)=1,
$$
and
$$
\PP\lk( \int_0^t \int_{\{x\in S:|{\mathfrak c}_t(x,s)|\ge 1\}} |{\mathfrak c}_t(x,s)|\,\nu(dx)\,ds<\infty\rk)=1.
$$
and proceed as above with $p=1$.
\del{

\medskip
Thanks to \cite[Theorem 4.1]{whitt}
the function $\Phi:C(\RR^ +_0;\RR )\times \DD (\RR^ +_0;\RR )\ni (u,v)\mapsto u+v\in \DD(\RR^ +_0;\RR)$  is continuous, therefore, the process
$$\RR^ +_0\ni t\mapsto \CK_{\mathfrak{A}}(u,\eta)(t)$$
has a  \cadlag version in $\RR$.
}

\medskip
Summing up, it follows that, if $(u,\eta)$ and $(\bar u,\bar \eta)$ have the same law on $\DD(\INT ;Y)\times \CMM$, then $\CK_{\bar{ \mathfrak A}}(\bar u,\bar \eta)(t),\bar u,\bar \eta)$ and $(\CK_{\mathfrak{A}}(u,\eta)(t),u, \eta)$ have the same law on $Y\times \DD(\INT;Y)\times \CMM$.
Since for all $t\in[0,T]$
$$\bar \PP\lk( \CK_{\bar {\mathfrak A} }(\bar u,\bar \eta)(t) -\bar u(t)=0\rk)=1,
$$
it follows that
$$ \PP\lk( \CK_{ \mathfrak A}(u,\eta)(t)- u(t)=0\rk)=1.
$$
In particular, the six tuple $( \Omega, \PP,\CF,\BF=( \CF_t)_{t\ge 0}, u, \eta)$ is a solution to \eqref{eq}.

\end{proof}

{
\section{Uniqueness}

Throughout this section, the notation of Section \ref{sec_abstr} will be kept. We are going to prove here that the abstract result of Kurtz \cite{kurtz1} can be applied to the problem \eqref{eq}. Or, in other words, that pathwise uniqueness for the equation \eqref{eq} implies joint uniqueness in law {and strong existence} for the equation \eqref{eq}.

In order to show this, let use define the Kurtz's {\it compatibility structure} and {\it $\CC$-compatibility} according to \cite[Definition 3.3]{kurtz1}.

{Throughout this section we fix an intensity measure $\nu$ and sets $S_n\in\mathcal S$ such that $S_n\uparrow S$ and $\nu(S_n)<\infty$.}

\begin{definition} Let us denote
$$
Z_1=\Bbb D([0,T],Y),\qquad Z_2={M_{\bar{\Bbb N}}(\{S_n\times{\Int}\})}\times X,
$$
$$
\CB^{Z_1}_t=\sigma(\pi_s:s\le t),\qquad\CB^{Z_2}_t=\sigma(R_t)\otimes\CB(X),
$$
where $\pi$ and $R$ are the canonical mappings, $\pi: \Bbb D([0,T],Y)\ni\pi\mapsto \pi(t)\in  Y$,
$$
R_t:{M_{\bar{\Bbb N}}(\{S_n\times{\Int}\})}\to{M_{\bar{\Bbb N}}(\{S_n\times{\Int}\})}:\mu\mapsto\mu(\cdot\cap(S\times(0,t])),
$$
and denote by $\mathcal C$ the Kurtz compatibility structure $\{(\CB^{Z_1}_t,\CB^{Z_2}_t):t\in[0,T]\}$.
\end{definition}

\begin{definition} If $A$ is an $Z_1$-valued random variable {over some probability space $\mathfrak{A}=(\Omega,\CF,\PP)$}, $B$ an $Z_2$-valued random variable over $\mathfrak{A}$
and $t\in[0,T]$, then $\mathscr F^A_t$ and $\mathscr F^B_t$ are the coarsest $\sigma$-algebras such that the mappings
$$
A:(\Omega,\mathscr F^A_t)\to(Z_1,\CB^{Z_1}_t)\quad \mbox{and}\quad  B:(\Omega,\mathscr F^B_t)\to(Z_2,\CB^{Z_2}_t)
$$
are measurable.
\end{definition}

\begin{remark} If $A$ is an $Z_1$-valued random variable over $\mathfrak{A}$ and $B=(\eta,\xi)$ an $Z_2$-valued random variable over $\mathfrak{A}$, it is rather standard to see that
$$
\mathscr F^A_t=\sigma(A_s:s\le t)\quad \mbox{and}\quad\mathscr F^B_t=\sigma(\eta_t)\lor\sigma(\xi)
$$
hold for every $t\in[0,T]$, as defined in \eqref{etat}.
\end{remark}

\begin{definition}\label{def44} We say that $Z_1$-valued random variables $A_1,\dots,A_n$ are $\CC$-compatible with an $Z_2$-valued random variable $B$ provided that
\begin{equation}\label{compat1}
\Bbb E\,[h(B)|\mathscr F^{A_1}_t\lor\dots\lor\mathscr F^{A_n}_t\lor\mathscr F^B_t]=\Bbb E\,[h(B)|\mathscr F^B_t]\qquad\textrm{a.s.}
\end{equation}
holds for every $t\in[0,T]$ and every real bounded Borel measurable function $h$ on $Z_2$.
\end{definition}

\begin{remark} $\CC$-compatibility of random variables $A_1,\dots,A_n$ with a random variable $B$ is actually a property of the joint law of $(A_1,\dots,A_n,B)$, as follows from \cite[Remark 3.5]{kurtz1}. Hence we can introduce the notion of $\CC$-compatibility for Borel probability measures on $Z_1^n\times Z_2$, see \cite[Definition 3.6]{kurtz1} and Definition \ref{compmeas} below.
\end{remark}

\begin{definition}\label{compmeas} A probability measure on $Z_1^n\times Z_2$ is called $\CC$-compatible provided that, if any $(A_1,\dots,A_n,B)$ are distributed according to $\mu$ then $A_1,\dots,A_n$ are $\CC$-compatible with $B$ in the sense of Definition \ref{def44}.
\end{definition}

\begin{lemma}\label{anothercomp} Let $\eta$ be a time homogeneous Poisson random measure with the intensity measure $\nu$ on $S$ for some filtration $(\mathcal F_t)_{t\in\INT}$ and let $u_0$ be an $\mathcal F_0$-measurable $X$-valued random variable. Then $Z_1$-valued random variables $A_1,\dots,A_n$ are $\CC$-compatible with $B=(\eta,u_0)$ if and only if
$\mathscr F^{A_1}_t\lor\dots\lor\mathscr F^{A_n}_t\lor\mathscr F^B_t$ is $\Bbb P$-independent of $\sigma(\eta^t)$.
\end{lemma}

\begin{proof}
Since $\eta=\eta_t+\eta^t$, \eqref{compat1} holds if and only if
$$
\Bbb E\,[h(\eta^t)|\mathscr F^{A_1}_t\lor\dots\lor\mathscr F^{A_n}_t\lor\mathscr F^B_t]=\Bbb E\,[h(\eta^t)|\mathscr F^B_t].
$$
But $\sigma(\eta^t)$ and $\mathscr F^B_t=\sigma(\eta_t)\lor\sigma(u_0)$ are $\Bbb P$-independent by Lemma \ref{adapttimhom}, hence \eqref{compat1} holds if and only if $\mathscr F^{A_1}_t\lor\dots\lor\mathscr F^{A_n}_t\lor\mathscr F^B_t$ is $\Bbb P$-independent of $\sigma(\eta^t)$.
\end{proof}

Now we are ready to define a Kurtz convexity constraint $\Gamma_\nu$}, see \cite[page 958]{kurtz1}.

\begin{definition}\label{gambet} If $\mu$ is a Borel probability measure on $Z_1\times Z_2$ and a random vector $(u,\eta,u_0)$ over a probability space $(\Omega,\CF,\PP)$
 has the distribution $\mu$, we say that $\mu$ satisfies a convexity constraint {$\Gamma_\nu$} provided that
\begin{itemize}
\item[(a)] $u(0)=u_0$ almost surely,
\item[(b)] there exists a solution $(\bar\Omega,\bar\CF,\bar\BF, (\bar\CF_t)_{t\in\INT},\bar\PP,\bar u,\bar\eta)$ to the equation \eqref{eq}
{satisfying \ref{hyp_solution}}, in particular \eqref{hyp_solution-eqn}, such that {$\nu$} is the intensity measure of $\bar\eta$ and the law of $(\bar u,\bar \eta)$ coincides with the law of $(u,\eta)$.
\end{itemize}
\end{definition}

\begin{remark} Paradoxically, despite of the notion, we need not prove here that the Kurtz's convexity constraint {$\Gamma_\nu$} really defines a convex set of probability measures
 on $Z_1\times Z_2$. For us, {$\Gamma_\nu$} is viewed as a mere constraint with no convexity properties. We will explain more in the proof of Theorem \ref{thm513}.
\end{remark}

Finally, we define the Kurtz set {$\CS_{\Gamma_\nu,\CC,\Theta_\nu\otimes\beta}$}, see \cite[page 958]{kurtz1}.

\begin{definition}\label{kucose} Let $\beta$ be a Borel probability measure on $X$. We denote by {$\CS_{\Gamma_\nu,\CC,\Theta_\nu\otimes\beta}$} the set of probability measures $\mu$ on $Z_1\times Z_2$ such that
\begin{itemize}
\item[(a)] $\mu$ satisfies the convexity constraint {$\Gamma_\nu$},
\item[(b)] $\mu$ is $\CC$-compatible,
\item[(c)] $\mu(Z_1\times\cdot)={\Theta_\nu}\otimes\beta$ on $\CB(Z_2)$,
\end{itemize}
{where $\Theta_\nu$ is the probability measure introduced in Lemma \ref{poiss_uni_distr}.}
\end{definition}

\begin{remark} Observe that the condition (b) in Definition \ref{kucose} is superfluous as it follows from (a) due to Lemma \ref{anothercomp}. We however present Definition \ref{kucose} as it is, to be conformal with Kurtz's notation in \cite{kurtz1}.
\end{remark}

\begin{remark} The Kurtz set {$\CS_{\Gamma_\nu,\CC,\Theta_\nu\otimes\beta}$} is in fact convex. We however do not need the convexity in this paper so we do not  prove it either.
\end{remark}

We can now give a full description of the set {$\CS_{\Gamma_\nu,\CC,\Theta_\nu\otimes\beta}$}.

\begin{corollary} Let $\beta$ be a Borel probability measure on $X$. Then $\mu\in{\CS_{\Gamma_\nu,\CC,\Theta_\nu\otimes\beta}}$ if and only if
\begin{itemize}
\item there exists a solution $(\bar\Omega,\bar\CF,\bar\BF, (\bar\CF_t)_{t\in\INT},\bar\PP,\bar u,{\bar\eta})$ to the equation \eqref{eq}
{satisfying Hypothesis \ref{hyp_solution}},
\item {$\nu$} is the intensity measure of {$\bar\eta$},
\item $\mu$ is the law of $(\bar u,{\bar\eta},\bar u(0))$,
\item $\beta$ is the law of $\bar u(0)$.
\end{itemize}
\end{corollary}

\begin{theorem}\label{thm513}
Let $\beta$ be a Borel probability measure on $X$ and assume that
\begin{itemize}
\item there exists a solution $(\bar\Omega,\bar\CF,\bar\BF, (\bar\CF_t)_{t\in\INT},\bar\PP,\bar u,\bar\eta)$ to the equation \eqref{eq} {satisfying Hypothesis \ref{hyp_solution}} such that {$\nu$} is the intensity measure of $\bar\eta$ and $\beta$ is the law of $\bar u(0)$,
\item whenever $(\Omega,\CF,\BF,(\CF_t)_{t\in\INT},\PP,u^1,\eta)$ and $(\Omega,\CF,\BF,(\CF_t)_{t\in\INT},\PP,u^2,\eta)$ are solutions to the equation \eqref{eq} {satisfying Hypothesis \ref{hyp_solution}} such that {$\nu$} is the intensity measure of $\bar\eta$, $\beta$ is the law of $u^1(0)$ and $u^1(0)=u^2(0)$ a.s. then $u^1=u^2$ a.s.
\end{itemize}
Then there exists a Borel measurable mapping
$$
F:\CMM\times X\to\Bbb D([0,T];Y)
$$
depending on {$\nu$} and $\beta$ such that
\begin{itemize}
\item[(1)] if $(\Omega,\CF,\BF, (\CF_t)_{t\in \INT},\PP,u,\eta)$ is a solution to the equation \eqref{eq} {satisfying Hypothesis \ref{hyp_solution}} such that {\color{Green}$\nu$} is the intensity measure of {$\eta$} and  $\beta$ is the law of $u(0)$ then $u=F(\eta,u(0))$ a.s. and $u$ is adapted to the $\Bbb P$-augmentation of the filtration $(\sigma(\eta_t),\sigma(u(0)))_{t\in\INT}$,
\item[(2)] if $(\Omega,\CF,\BF, (\CF_t)_{t\in\INT},\Bbb P)$ is a stochastic basis, $\xi^*$ is an $X$-valued $\mathcal F_0$-measurable random variable with law $\beta$ and $\eta$ is a time homogeneous $(\mathscr F_t)_{t\in\INT}$-Poisson random measure with intensity {\color{Green}$\nu$} then $u=F(\eta,\xi)$ is $(\mathscr F^{\Bbb P}_t)_{t\in\INT}$-adapted, $u(0)=\xi$ a.s. and $(\Omega,\CF,\BF, (\CF^{\Bbb P}_t)_{t\in\INT},\Bbb P,u,\eta)$ is a solution to \eqref{eq} satisfying Hypothesis \ref{hyp_solution}.
\end{itemize}

Consequently, if $(\Omega^i,\CF^i,\BF^i, (\CF_t^i)_{t\in\INT},\PP^i,u^i,\eta^i)$, $i=1,2$ are solutions to the equation \eqref{eq} satisfying Hypothesis \ref{hyp_solution} such that {$\nu$} is the intensity measure of $\eta^1$ and $\eta^2$, $\beta$ is the law of $u^1(0)$ and $u^2(0)$ then the law of  $(u^1,\eta^1)$ coincides with the law of $(u^2,\eta^2)$.
\end{theorem}

\begin{proof}
By the assumptions in Theorem \ref{thm513}, the Kurtz set {$\CS_{\Gamma_\nu,\CC,\Theta_\nu\otimes\beta}$} is non-empty and pointwise uniqueness holds for $\CC$-compatible solutions of {$(\Gamma_\nu,\Theta_\nu\times\beta)$} in the sense of \cite[p. 959]{kurtz1}. Hence by the implication $(a)\Rightarrow(b)$ in \cite[Theorem 3.14]{kurtz1}, joint uniqueness in law holds for compatible solutions, i.e. {$\CS_{\Gamma_\nu,\CC,\Theta_\nu\otimes\beta}$} contains exactly one measure, and there exists a strong compatible solution in the sense of \cite[p. 959]{kurtz1} and \cite[Lemma 3.11]{kurtz1}, i.e. (1) holds.

To prove (2), once $u=F(\eta,\xi)$, we have that the law of $(u,\eta,\xi)$ coincides with the law of $(\bar u,\bar\eta,\bar u(0))$. Hence $u(0)=\xi$ a.s. and $u$ is compatible with $(\eta,u(0))$ by \cite[Remark 3.5]{kurtz1}. Thus $u$ is adapted to the $\Bbb P$-augmentation of the filtration $(\sigma(\eta_t)\lor\sigma(u(0)))_{t\in\INT}$ by \cite[Lemma 3.11]{kurtz1}. The rest then follows from Lemma \ref{transequal}.

We must point out here that the constraint {$\Gamma_\nu$} need not define a convex subset of Borel probability measures on $Z_1\times Z_2$ if we apply just the implication $(a)\Rightarrow(b)$ in \cite[Theorem 3.14]{kurtz1}. The convexity of the Kurtz set {$\CS_{\Gamma_\nu,\CC,\Theta_\nu\otimes\beta}$} is needed just for the implication $(a)\Leftarrow(b)$ in \cite[Theorem 3.14]{kurtz1} which we do not apply in our case.
\end{proof}

\appendix

\section{Uniqueness of the stochastic integral}

Let $X$ and $E$ be two separable Banach. Later on we will take $X$ to  be one of the
spaces $E$ or  $L^p(S,\nu,E)$. Let $\mathfrak{A}=(\Omega,\CF,(\CF_t)_{t\in\INT},\PP)$
be an arbitrary filtered probability space and $\eta$ be a Poisson random measure define over $\mathfrak{A}$.
Let
$\mathcal{N}(\Omega\times \INT ;X)$ be the space of (equivalence classes of)
progressively measurable functions $\xi :\Omega\times \INT\to X$.

For $ q\in (1,\infty ) $ we  set
\begin{eqnarray}\;\;
\\ \nonumber
\mathcal{N}^q(\Omega\times \INT,\mathcal{F};X)&=&
\left\{
\xi \in \mathcal{N}(\Omega\times \INT ,\mathcal{F};X): \;
 \int_0^\infty\vert\xi (t)\vert^q\,dt<\infty \mbox{ a.s. }
\right\},
\label{def:Nq}
\\
\\ \nonumber
\;\; \;\;
\mathcal{M}^q(\Omega\times \INT,\mathcal{F};X)&=&
\left\{\xi \in \mathcal{N}(\Omega\times \INT,\mathcal{F};X):\mathbb{E}\int_
0^\infty\vert\xi (t)\vert^q\,dt<\infty\right\}.
\label{def:Mq}
\end{eqnarray}

Let $E$ be a space of martingale type $p$ and put $X=L^p(E;\nu)$, and
let $\xi\in \CN^ p(\Omega\times \INT,\CF;X)$. In particular, $\xi: \Omega\times \INT \times S\to E$ be a progressively measurable  process such that $\PP$--a.s.\
\DEQSZ
\label{E:Main_cond-1.01}
 \int_0^T \int_S\vert \xi(r,z)\vert_E^p\,
\nu(dz)\,dr %= \int_0^T \vert \xi(r)\vert_X^p\, dr
  &< &\infty.
\EEQSZ

Let us consider the law of the triplet $(\eta,\xi,I)$, where $I$ is the It\^o  integral of $\xi$ with respect to
$\eta$ as defined on page \pageref{eqn-2.02}. %In \cite{zdzandme} we were interested, weather the law is unique.
We have shown in Theorem 2.4 of \cite{zdzandme} that the law is unique in case
\DEQS
\EE \int_0^T  \int_S\vert \xi(r,z)\vert_E^p\,
\nu(dz)\,dr %= \int_0^T \vert \xi(r)\vert_X^p\, dr
  &< &\infty.
\EEQS
In this appendix we want to extend this result to all progressively processes
$\xi$ satisfying only \eqref{E:Main_cond-1.01}.
However, before stating  the Theorem we want to define  uniqueness in law.

\begin{definition}\label{def-equality}  %(compare \cite[Theorem 5.9]{para})
Let $(X,\CX)$ be a measurable space.
When we say that $\xi_1$ and $\xi_2$ have the same law on $X$ (and write $\CL aw(\xi_1)=\CL aw(\xi_2)$ on $X$), we mean that $\xi_i$, $i=1,2$, are $X$-valued random variables defined over some
probability spaces $(\Omega_i,\CF_i,\PP_i)$, $i=1,2$, such that
$$
\PP_1 \bar{\circ}\, \xi_1= \PP_2\bar{\circ}\, \xi_2,
$$
where $\PP_i\bar{\circ}\, \xi_i(A) = \PP_i(\xi ^ {-1}_i(A))$, $A\in\CX$, $i=1,2$, is a   probability measure on $(X,\CX)$ called the law of $\xi_i$.
\end{definition}

\begin{theorem} \label{equality}
Let $(\Omega_i,\CF_i,\PP_i)$, $i=1,2$, be two probability spaces and $(\CF ^ i_t)_{t\in\INT}$ a filtration of $(\Omega_i,\CF_i)$.
 % the first $\Omega_i$ one is filtered by $\CF ^ 1_t$.
Assume that $\{(\xi _i,\eta _i),n\in\NN\}$, $i=1,2$, are two $L ^ p(\INT;L^p(S,\nu,E)) \times  \CMM $ valued random variables
defined on $(\Omega_i,\CF_i,\CF ^ i_t,\PP_i)$, $i=1,2$, respectively.
Assume that $\eta_1$ is a time homogeneous Poisson random measure over $(\Omega_1,\CF_1,(\CF ^ 1_t)_{t\in\INT},\PP_1)$
with intensity $\nu $.
Furthermore, assume that $\xi_1\in\CN(\Omega_1\times \INT;L^p(S,\nu,E))$ with respect to $(\CF ^ 1_t)_{t\in\INT}$.

\noindent
Let
\DEQS
I _i (t): =I(\xi_i,\eta_i)(t)= \int_0 ^ t % \infty
\int_S %1_{(0,t]}(s)\,
  \xi_i(s,z)\, \tilde  \eta_i(dz,ds).% \to  \int_0 ^ t \int_S \xi(s,z)\tilde  \eta_i(dz,ds),\quad i=1,2.
\EEQS
\begin{trivlist}
\item[(i)]
If $\CL aw( (\xi  _1,\eta_1))=\CL aw((\xi _2,\eta_2)) $ on
$$L ^ p(\INT;L ^ p(S,\nu;E)) \times \CMM ,
$$
then
$\CL aw( (I _1,\xi  _1,\eta_1))=\CL aw((I _2,\xi _2,\eta_2)) $ on
$$\DD(\INT;E)\times  L ^ p(\RR_+;L ^ p(S,\nu;E))\times  \CMM .
$$
\item[(ii)]
%
%[0.2cm]
If $\CL aw( (\xi  _1,\eta_1))=\CL aw((\xi _2,\eta_2)) $ on
$$L ^ p(\INT;L ^ p(S,\nu;E))\times  \CMM ,
$$
then
$\CL aw( (I _1,\xi  _1,\eta_1))=\CL aw((I _2,\xi _2,\eta_2)) $ on
$$L^p(\INT;E)\times  L ^ p(\INT;L ^ p(S,\nu;E))\times  \CMM .
$$
\end{trivlist}
\end{theorem}

\begin{proof}
In fact, Theorem \ref{equality} follows from Theorem 2.4 in \cite{zdzandme} by localization.
First, for a fixed $R>0$ let us first introduce the stopping times
$$
\tau^R_i=\inf_{t>0} \lk\{ \int_0^t|\xi(s)|_X^p \, ds \ge R\rk\}.
$$
Put $\xi^R_i := 1_{[0,\tau]} \, \xi$. Observe, using the shifted Haar projection defined in \eqref{haar-projection-shifted}
one obtains a sequence of simple functions $\{\fhs_n \xi^R_i:n\in\NN\}$ such that
$\PP^i$--a.s. $\fhs_n \xi^R_i\to \xi^R_i$ in $L^p(\INT;E)$. In addition
$$
\EE  \int_0^T \int_S\vert \xi_i^R(r,z)\vert_E^p\,
\nu(dz)\,dr\le R.
$$
Thus, Theorem 2.4 in \cite{zdzandme} is applicable and we have
$
\Law( (I ^R_1,\xi ^R _1,\eta_1))=\Law((I ^R_2,\xi^R _2,\eta_2))$ on % \mbox{ on }
$$
(\DD(\INT;X)\cap L^p(\INT ;X))\times  L^p(\INT;L^p(Z,\nu ;E))\times \CMM ,
$$
where
\begin{equation}  \label{E:Main_stoch_conv}
 I^R_i(t) = \int_0 ^ t \int_S
\xi^R_i(s,z)\, \tilde \eta_i(ds,dz), \quad t\in\INT.
\end{equation}

Let
$$
A_i^R:= \lk\{ \omega\in\Omega_i: \int_0^T \lk|\xi_i(s) \rk|_{L^p(Z,\nu ;E)}^p \, ds\le R\rk\}.
$$
Then, first,
on $A^R_i$, $\xi_i^R=\xi_i$, secondly, $A^ {R_1}\supset A^ {R_2}$ for $R_1>R_2$, and, thirdly, by Lemma 1.14 \cite{kallenberg},
 $\lim_{R\to\infty} \PP\lk( A^R_i\rk)=\PP\lk(\Omega\rk)=1$.
%  (this follows by Corollary 1, p. 52, pataseraty - find a direct one).
Take a set
\DEQS
\lqq{ B_1\times B_2\times B_3 \in \CB(\DD(\INT;X)\cap L^p(\INT ;B))}
&& \\
& &\times \CB( L^p(\INT;L^p(Z,\nu ;E)))\times\CB( \CMM).
\EEQS
Since $\xi^ R_i\le \xi_i$, we have by the dominated convergence Theorem
\DEQS
\lqq{
\PP_1\lk( (u_1,\xi_1,\eta)\in B_1\times B_2\times B_3\rk)} &&
\\
& = & \lim_{R\to \infty} \PP_1\lk( (u_1,\xi_1,\eta)\in B_1\times B_2\times B_3, \tau_1^R>t\rk)
% \PP^1\lk(\tau_1^R>t\rk)
%+\lim_{R\to \infty} \PP_1\lk( (u_1,\xi_1,\eta)\in B_1\times B_2\times B_3\mid \tau_1^R<t\rk)
\\
& = & \lim_{R\to \infty} \PP_1\lk( (u ^R_1,\xi^R_1,\eta)\in B_1\times B_2\times B_3\rk)
\\
& = & \lim_{R\to \infty} \PP_2\lk( (u ^R_2,\xi^R_2,\eta)\in B_1\times B_2\times B_3\rk)
\\
& = & \lim_{R\to \infty} \PP_2\lk( (u ^R_2,\xi^R_2,\eta)\in B_1\times B_2\times B_3, \tau_1^R>t\rk)
\\
& = & \PP_2\lk( (u _2,\xi_2,\eta)\in B_1\times B_2\times B_3\rk).
\EEQS
Now, the assertion follows.
\end{proof}

\del{
\begin{assumption}\label{det_cont}
The convolution operator $\CT:L^ p([0,T];E)\to C([0,T];X)$ defined by
$$
(\CT u)(t) := \int_0^ t e^ {-(t-s)A}\, F(u(s))\, ds ,\quad u\in L^ p([0,T];E),
$$
is continuous.
\end{assumption}
\begin{corollary}

\end{corollary}
}

\section{The Haar Projection}

\subsection{The Haar projection onto $L^ q$--spaces}

For $n\in\mathbb{N}$, let  $\Pi^n=\{ s^n_0=0<s^n_1<\cdots <s^n_{2^n}\}$
be a  partition of the interval $[0,T]$ defined  by $s_j ^ n=j\,2 ^ {-n}T$, $j=1,\cdots, 2^n$.
Each interval of the form $(s_{j-1} ^ n,s_j ^ n]$, where $n\in \mathbb{N}$ and $j=1,\ldots, 2^n$ is called a \it dyadic \rm interval.
For  $n\in \mathbb{N}$, the \it $j^ {th}$ element, for $j=1,\ldots, 2^n$,  of the Haar system of order $n$ \rm is the indicator function of the interval $(s_{j-1}^n,s_j^{n}]$, i.e.\
$1_{(s_{j-1}^{n},s_j^{n}]}$.
First, given a function $x:[0,T]\to Y$, $Y$ a Banach space, let us define
the averaging operator  $\iota_{j,n}:  L ^ p(0,T,Y)\to Y$ over the interval $(s^n_{j-1},s^n_{j}]$
by
\begin{equation}  \label{xdef-ij}
 \iota _{j,n}(x) :=
\frac{1}{s^n_{j}-s^n_{j-1}}\; \int_{s^ n_{j-1}}^{s ^ n_{j}} x(s)\,ds, \;  x\in  L ^ p([0,T],Y).
\end{equation}
For $n\in\mathbb{N}$, let $\fhs_n: L ^ p([0,T],Y)\to L ^ p([0,T],Y)$ be the \it shifted Haar projection \rm of order $n$, i.e.
\DEQSZ\label{haar-projection-shifted}
\fhs_n  x =\sum_{j=1} ^{2^n -1} 1_{(s^n_j,s^n_{j+1}]}\otimes \iota_{j,n}(x), \;\; x\in L ^ p(\INT;Y),
\EEQSZ
where we put ${\iota }_{0,n}=0$ for every $n\in\mathbb{N}$.
In the above, for $f\in L ^ p([0,T],\mathbb{R})$ and $y\in  Y$, by $f\otimes y$ we mean an element of $L ^ p([0,T],Y)$ defined by
$[0,T]\ni t\mapsto f(t) y\in Y$. For completeness, let us cite the following results taken from \cite[Appendix B]{uniquness1}.
\begin{proposition}
\label{cont-haar-shift}
The following holds:
\begin{trivlist}
%\label{cont-haar-shift}
\item[(i)]
For any $n\in\mathbb{N}$, the shifted Haar projection $\fhs_n:L ^ p([0,T] ;Y)\to L ^ p([0,T];Y)$ is a continuous operator.
\item[(ii)]
For all $x\in  L ^ p([0,T] ;Y)$, $\fhs_n x\to x$ in $L ^ p([0,T] ;Y)$.
\end{trivlist} %{enumerate}
\end{proposition}

\begin{remark}\label{shseq}
Observe, for any $\xi\in\CN_p([0,T] ;X)$, the process $[0,T ]\ni t\mapsto \fhs_n x(t)$ is simple, left continuous and predictable and
the sequence $\{ \fhs_n \xi:n\in\NN\}$ converges to $\xi$ in $L ^p([0,T];Y)$.
\end{remark}

\subsection{The  Haar projection onto the Skorohood space}

If the underlying space is the Skorohood space, the Haar projection have to defined by another way.
For the Skorohood space we refer to  Billingsley \cite{billingsley}, Ethier and
Kurtz \cite{MR838085} and Jacod and Shiryaev
\cite{jacod}.

Let $(Y,{|\cdot |}_{Y})$ be a separable Banach space.
The space $\DD([0,1];Y)$ denotes the space of all right continuous functions $x:[0,1]\to Y$
with left-hand limits. Let $\Lambda $ denote the class of all strictly increasing continuous functions $\lambda:[0,1]\to [0,1]$ such that $\lambda(0)=0$ and {$\lambda(1)=1$}.
Obviously any element $\lambda \in \Lambda$ is a homeomorphism of $[0,1]$ onto itself.
Let us define the Prohorov metric $d_0$ by
\begin{equation}\label{eqn-dlog}
\begin{array}{rcl}
\Vert \lambda\Vert_{\rm log} &:=& \sup_{t\not=s\in [0,1]} \Big\vert \log \frac{\lambda(t)-\lambda(s)}{t-s}\Big\vert,\\
\Lambda_{\rm log}&:=& \big\{ \lambda \in \Lambda:  \Vert \lambda\Vert_{\rm log} <\infty\big\}\\
d_0(x,y)&:=&\inf \Big\{ \Vert \lambda\Vert_{\rm log} \vee  \sup_{t\in [0,1]} {|x(t)-y(\lambda(t))|}_{Y}: \, \lambda \in \Lambda_{\rm log}\Big\}.
\end{array}
\end{equation}

The space $\DD([0,1];Y)$ equipped with the metric $d_0$ is a separable complete metric space.
Here, in this section we shortly introduce
the so called 'dyadic projection' onto the Haar system. For more detailed information we refer to \cite{uniquness1}.
\begin{definition} Assume that $n\in \mathbb{N}^\ast$.
The $n$-th order  dyadic projection is an operator $\fh^ \DD_n: \DD(\INT ;E)\to\DD(\INT ;E)$,  such that
$\fh^ \DD_n x$, $x\in \DD(\RR_0^+;E)$, is defined by
\DEQSZ\label{dyadic}
(\fh^ \DD_n x)(t) := \sum_{i=0}^\infty \,1_{(2 ^ {-n}i,2 ^ {-n}(i+1)]}(t) \; x( 2 ^ {-n}i), \;\; t\in \INT.
\EEQSZ
\end{definition}

An  important  property  of  the dyadic projection is  given
in the following result.

\begin{proposition}\label{chap:jacod}
The following holds.
\begin{enumerate}
\item  If $x\in\DD(\INT ;E)$ then  $\lim_{n\to\infty} d_0(x,\fh^ \DD_n  x)=0$;
\item if $K\subset \DD(\INT ;E)$ is  compact, then
{$$
\lim_{n\to\infty} \,\sup_{x\in K}\, d_0(x,\fh^ \DD _n x)=0.
$$}
\end{enumerate}
\end{proposition}
Finally, we need the fact, that the integral operator is continuous operator on $\DD(\INT ,E)$.
\begin{proposition}\label{integral-skorohod}
If $(x_n)\to x$ in  $\DD(\RR_0^+,E)$, then for all $s,t\in\Int $ we have
$$
\int_s ^ t  x_n(r)\, dr \to \int_s ^ t x(r)\, dr
$$
in $E$ as $n\to\infty$.
\end{proposition}

{
\section{Polish measure spaces}\label{pomesp}

Let $(S,\CS)$ be a Polish space and let $S_n\in\CS$ satisfy $S_n\uparrow S$. Then there exists a metric $\varrho$ on $S$ such that $(S,\varrho)$ is a complete separable metric space, $\mathscr B(S,\varrho)=\CS$ and $S_n$ is closed for every $n\in\Bbb N$, see e.g. \cite[(13.5) page 83]{kechris}. Consider the L\'evy-Prokhorov metric on $M_+(S,\varrho)$
$$
\pi(\mu,\nu)=\inf\,\{\varepsilon>0:\mu(A)\le\nu(A^\varepsilon)+\varepsilon,\,\nu(A)\le\mu(A^\varepsilon)+\varepsilon,\,\forall A\in\CS\}
$$
where $A^\varepsilon=\{x\in S:\exists a\in A,\,\varrho(x,a)<\varepsilon\}$. Then $(M_+(S,\varrho),\pi)$ is a complete separable metric space and $\pi(\mu_n,\mu)\to 0$ iff
$$
\int_Sf\,d\mu_n\to\int_Sf\,d\mu,\qquad\forall f\in C_b(S,\varrho),
$$
see e.g.\ \cite[page 72-73]{billingsley}, where the proof for probability measures can be quite easily adapted to finite non-negative measures.

\begin{lemma}\label{lc1} The $\sigma$-algebra $\mathcal M_+(S)$ on $M_+(S)$ generated by the mappings $\mu\mapsto\mu(A)$, $A\in\CS$ coincides with the Borel $\sigma$-algebra $\mathscr B(M_+(S),\pi)$.
\end{lemma}

\begin{proof} The mapping $\mu\mapsto\mu(A)$ is upper semicontinuous on $(M_+(S),\pi)$ for every $A\subseteq S$ closed and lower semicontinuous for every $A\subseteq S$ open, hence Borel measurable for every $A\in\CS$. In particular, $\mathcal M_+(S)\subseteq\mathscr B(M_+(S),\pi)$. On the other hand, let $\mathcal G$ be a countable basis of open sets in $(S,\varrho)$ closed under finite unions. Then
$$
\{\mu:\,\pi(\mu,\theta)<r\}=\bigcup_{\varepsilon\in\Bbb Q\cap(0,r)}\bigcap_{A\in\mathcal G}\{\mu:\,\mu(A)\le\theta(A^\varepsilon)+\varepsilon,\,\theta(A)\le\mu(A^\varepsilon)+\varepsilon\}\in\mathcal M_+(S).
$$
Hence open balls in $(M_+(S),\pi)$ belong to $\mathcal M_+(S)$ and since $(M_+(S),\pi)$ is separable, every open set is a countable union of open balls. Consequently, every open set in $(M_+(S),\pi)$ belong to $\mathcal M_+(S)$, hence $\mathscr B(M_+(S),\pi)\subseteq\mathcal M_+(S)$.
\end{proof}

\begin{lemma}\label{lc2} The set of integer-valued measures $M_{\Bbb N}(S)$ is closed in $(M_+(S,\varrho),\pi)$.
\end{lemma}

\begin{proof}
Let $\pi(\mu_n,\mu)\to 0$, $\mu_n$ be integer-valued and $k<\mu(A)<k+1$ for some integer $k$ and some $A\in\CS$. By regularity, we can find a compact $C\subseteq A$ such that $k<\mu(C)\le\mu(A)<k+1$ and $\delta>0$ such that $k<\mu(C)-\delta$ and $\mu(C^{2\delta})+\delta<k+1$. If $\pi(\mu_n,\mu)<\delta$ then
$$
k<\mu(C)-\delta\le\mu_n(C^\delta)\le\mu(C^{2\delta})+\delta<k+1,
$$
which cannot happen as $\mu_n(C^\delta)$ is an integer.
\end{proof}

\begin{lemma}\label{lc3} $(M_{\bar \NN}(\{S_n\}),\CM_{\bar \NN}(\{S_n\}))$ is a Polish space.
\end{lemma}
\begin{proof}
Consider the metric
$$
\rho(\mu,\nu)=\sum_{n=1}^\infty 2^{-n}\min\,\{1,\pi(\mu(\cdot\cap S_n),\nu(\cdot\cap S_n))\},\qquad\mu,\nu\in M_{\bar \NN}(\{S_n\}).
$$
Then $(M_{\bar \NN}(\{S_n\}),\rho)$ is a metric space and the mapping
$$
I:(M_{\bar \NN}(\{S_n\},\rho)\to(M_{\Bbb N}(S),\pi)^{\Bbb N}:\mu\mapsto(\mu(\cdot\cap S_n))_{n\in\Bbb N}
$$
is a homeomorphism onto a closed set in $(M_{\Bbb N}(S),\pi)^{\Bbb N}$, hence $(M_{\bar \NN}(\{S_n\}),\rho)$ is a complete separable metric space by Lemma \ref{lc2}. To show closedness, let $\lim_{j\to\infty}\pi(\mu_j(\cdot\cap S_n),\theta_n)=0$ for every $n\in\Bbb N$, and let $m<k$. If $f\in C_b(S_m)$ and we extend $f$ by zero on $S\setminus S_m$ then $f\in C_b(S)$ since $S_m$ is clopen. So $\partial S_n=\emptyset$, % and
$$
\theta_n(S\setminus S_n)=\lim_{j\to\infty}\mu_j((S\setminus S_n)\cap S_n)=0,\qquad\forall n\in\Bbb N,
$$
and
$$
\int_Sf\,d\theta_k=\lim_{j\to\infty}\int_Sf\,d\mu_j(\cdot\cap S_k)=\lim_{j\to\infty}\int_Sf\,d\mu_j(\cdot\cap S_m)=\int_Sf\,d\theta_m
$$
so $\theta_m(\cdot)=\theta_k(\cdot\cap S_m)$. In particular, $\theta_n(A)\uparrow$ for every $A\in\CS$ and $\theta(A)=\lim_n\theta_n(A)$ is a $\sigma$-additive, $\bar{\Bbb N}$-valued measure on $\CS$ and $\theta_m(\cdot)=\theta(\cdot\cap S_m)$.

Now, by Lemma \ref{lc1}, the mapping $\mu\mapsto\mu(A\cap S_k)$ is $\mathscr B(M_{\bar \NN}(\{S_n\}),\rho)$ measurable for every $k\in\Bbb N$ and every $A\in\CS$ since $I$ is Borel measurable. Hence $\CM_{\bar \NN}(\{S_n\})\subseteq\mathscr B(M_{\bar \NN}(\{S_n\}),\rho)$. On the other hand, the mapping
$$
(M_{\bar \NN}(\{S_n\}),\CM_{\bar \NN}(\{S_n\}))\to(M_{\Bbb N}(S),\mathscr B(M_{\Bbb N}(S))):\mu\mapsto\mu(\cdot\cap S_k)
$$
is measurable for every $k\in\Bbb N$ by Lemma \ref{lc1}. So, if $\theta\in M_{\bar \NN}(\{S_n\})$ is fixed, the mapping
$$
(M_{\bar \NN}(\{S_n\}),\CM_{\bar \NN}(\{S_n\}))\to\Bbb R:\mu\mapsto\pi(\mu(\cdot\cap S_k),\theta(\cdot\cap S_k))
$$
is measurable for every $k\in\Bbb N$. Consequently, the mapping
$$
(M_{\bar \NN}(\{S_n\}),\CM_{\bar \NN}(\{S_n\}))\to\Bbb R:\mu\mapsto\rho(\mu,\theta)
$$
is measurable. Since $(M_{\bar \NN}(\{S_n\}),\rho)$ is a separable metric space and every open set is a countable union of open balls, we conclude that $\mathscr B(M_{\bar \NN}(\{S_n\}),\rho)\subseteq\CM_{\bar \NN}(\{S_n\})$.
\end{proof}

}

\end{document}